\documentclass[a4paper,11pt]{article}
\usepackage{amsmath,amssymb,amsthm}
\usepackage[latin1]{inputenc}
\usepackage[T1]{fontenc}
\usepackage{graphicx}
\usepackage{mathrsfs}
\usepackage{dsfont}
\usepackage{natbib}

\newtheorem{Th}{Theorem}

\newtheorem{Lemma}{Lemma}
\newtheorem{Hyp}{Assumption}

\def\1{\mathds{1}}
\def\var{\mathrm{Var}}

\renewenvironment{proof}{\noindent{\bf Proof.}}{\hfill $\Box$\par\noindent}

\newcommand{\E}{\ensuremath{\mathbb{E}}}
\renewcommand{\P}{\ensuremath{\mathbb{P}}}
\newcommand{\R}{\ensuremath{\mathbb{R}}}
\renewcommand{\L}{\ensuremath{\mathbb{L}}}
\newcommand{\C}{\ensuremath{\mathcal{C}}}
\renewcommand{\L}{\ensuremath{\mathbb{L}}}
\newcommand{\e}{\ensuremath{\varepsilon}}
\renewcommand{\S}{\ensuremath{\mathbb{S}}}

\newcommand{\eps}{\varepsilon}

\title{Goodness-of-fit test for noisy directional data}
\author{Claire \textsc{Lacour} \thanks{Laboratoire de Math\'ematique, UMR 8628, Universit\'e Paris Sud, 91405 Orsay Cedex, France, Email: claire.lacour@math.u-psud.fr}\and Thanh Mai \textsc{Pham Ngoc} \thanks{Laboratoire de Math\'ematique, UMR 8628, Universit\'e Paris Sud, 91405 Orsay Cedex, France  Email: thanh.pham\_ngoc@math.u-psud.fr}}

\begin{document}

\maketitle

\begin{abstract}
We consider spherical data $X_i$ noised by a random rotation $\varepsilon_i\in$ SO(3) so that only the sample $Z_i=\varepsilon_iX_i$, $i=1,\dots, N$ is observed. We define a nonparametric test procedure to distinguish $H_0:$ ''the density $f$ of $X_i$ is the uniform density $f_0$ on the sphere'' and $H_1:$ ''$\|f-f_0\|_2^2\geq \C\psi_N$ and $f$ is in a Sobolev space with smoothness $s$''. 
For a noise density $f_\varepsilon$ with smoothness index $\nu$, we show that an adaptive procedure (i.e. $s$ is not assumed to be known) cannot have a faster rate of separation than $\psi_N^{ad}(s)=(N/\sqrt{\log\log(N)})^{-2s/(2s+2\nu+1)}$ and we provide a procedure which reaches this rate.
We also deal with the case of super smooth noise. We illustrate the theory by implementing our test procedure for various kinds of noise on SO(3) and by comparing it to other procedures. Applications to real data in astrophysics and paleomagnetism are provided. 

\end{abstract}



\section{Introduction}
We consider the spherical convolution model. We observe:
\begin{equation}\label{model}
Z_i=\e_iX_i,\quad i=1,\ldots,N
\end{equation}
where the $\e_i$ are i.i.d.  random variables of SO(3) the rotation group in $\R^3$ and the $X_i$'s are i.i.d. random variables on $\S^2$, the unit sphere in $\R^3$. We suppose that $X_i$ and $\e_i$  are independent. We also assume that the distributions of $Z_i$ and $X_i$ are absolutely continuous with respect to the uniform measure on $\S^2$ and we set $f_Z$ and $f$  the densities of $Z_i$ and $X_i$ respectively.  The distribution of $\e_i$  is absolutely continuous with respect to the probability Haar measure on SO(3) and we will denote the density of the $\varepsilon_i$'s by $f_\varepsilon$. 
Then we have
$$f_Z=f_\e*f,$$
where $*$ denotes the convolution product which is defined below in (\ref{convolution}).

Roughly speaking, the spherical convolution model provides a setup where each genuine observation $X_i$ is contaminated by a small random rotation.
The aim of the present paper is to provide a nonparametric adaptive minimax goodness-of-fit testing procedure on $f$ from the noisy observations $Z_i$. More precisely, let $f_0$ being the uniform density on $\S^2$, we consider the problem of testing the null hypothesis $f=f_0$ with alternatives expressed in $\mathbb{L}_2$ norm over Sobolev classes. In this work, we only deal with the case of the uniform density. The tools developed in the proofs are specific to the uniform distribution and proved to be already quite technical. The cases of other fixed densities are beyond the scope of the paper. Furthermore as explained below, testing the uniform distribution is already of great interest in practice. 

Spherical data arise in many areas of scientific experimentation and observation. As examples of directional data from various fields, we instance in astrophysics the arrival directions of the Ultra High Energy Cosmic rays (UHECR), from structural geology the facing directions of conically folded planes, from paleomagnetism the measurements of magnetic remanence in rocks, from meteorology the observed wind directions at a given place and from physical oceanography the measurements of current ocean directions.  In this work, we will particularly focus on the UHECR study and paleomagnetism as some applications of our statistics procedure. 

In astrophysics, for instance, a burning issue consists in understanding the behaviour of the so-called Ultra High Energy Cosmic Rays (UHECR). These latter are cosmic rays with an extreme kinetic energy (of the order of $10^{19}$ eV) and the rarest particles in the universe. The source of those most energetic particles remains a mystery and the stake lies in finding out their origins and which process produces them. Astrophysicists have at their disposal directional data which are measurements of the incoming directions of the UHECR on Earth. Needless to say that  finding out more about the law of probability of those incoming directions is crucial to gain an insight into the mechanisms generating the UHECR.   \cite{FDKP12} recently developed isotropy goodness-of-fit tests based on the so-called needlets for the non perturbated case. Their study is focused on the practical aspect with nice simulations connected to realistic cosmic rays scenarios. But the difficulty lies in the fact that the observed UHECR do 
not come necessarily from the genuine direction as specified by  \cite{FDKP12}.  Their trajectories are deflected by Galactic and intergalactic fields. As this deflection is inevitable in the measurements, it is quite challenging and essential to take into account this uncertainty in the statistical modelling. A first way to model the deflection in the incoming directions can be done thanks to the model (\ref{model}) with random rotations.


 Concerning the hypotheses about the underlying probability of the incoming directions, several are made. A uniform density would suggest that the UHECR are generated 
by cosmological effects, such as the decay of relic particles from the Big Bang. On the contrary, if these UHECR are generated by astrophysical phenomena (such as acceleration into Active Galactic Nuclei (AGN)), then we should observe a density function which is highly non-uniform
and tightly correlated with the the local distribution of extragalactic supermassive black holes at the center of nearby galaxies (AGN).  First results seemed to favour a non-uniform density but as underlined by \cite{FDKP12}, a more recent analysis based on 69 observations of UHECR softens this conclusion of anisotropy. To this prospect, these relevant considerations lead naturally to goodness-of-fit testing on the uniform density in the noisy model (\ref{model}).

Considering goodness-of-fit testing in the spherical convolution model not only finds its interest in the above important applications, but it also fills a gap both in the noisy setup testing literature and the spherical convolution one. Indeed, convolution models have been extensively studied in the Euclidean setting (see for instance \cite{fan91}, \cite{PenskyVida99}, \cite{CRT}, \cite{Butuceatsybakov}, \cite{Kalifa}, \cite{Meister09} and references therein), and more recently 
in other geometric frameworks, like the hyperbolic plane (see \cite{Huckemann}) or the sphere. However, so far, only estimation has been treated in the spherical setup. For the nonparametric estimation problem, one is interested in recovering the underlying density $f$ from noisy observations $Z_i$. The pioneer works of \cite{HealyHendriksKim},
\cite{kimkoo02}, \cite{kimkoo04} introduced a minimax estimation procedure based on the Fourier basis of $\L_2(\S^2)$. Recently, \cite{KPP11} proposed an optimal and adaptive hard thresholding estimation procedure based on needlets. 

Nonparametric goodness-of-fit testing has aroused a lot of interest.  For minimax testing, we refer to the work of \cite{Ingster93} which is the main reference in the field. \cite{Spok96} first established adaptive testing procedure based on wavelets over Besov bodies. Nonetheless, goodness-of-fit testing has mainly focused on the case of direct observations. Indeed, very few works have been devoted to the case of indirect observations. Let us cite the works of \cite{Bissantz2009} for the inverse regression problem and \cite{Holzmann} for the multivariate convolution density model.  \cite{butucea07} built minimax nonparametric goodness-of-fit testing for convolution models based on kernels methods and \cite{butuceamatiaspouet09} made a step forward by building an adaptive testing procedure in the noisy setup.

 We would also like to bring to the reader's attention some interesting facts when encountering testing problems with indirect observations.  Indeed, there is a natural connection between the following approaches : to test $f=f_0$ or to test $f_\e*f= f_\e*f_0$.
This question has been the object of the recent work of \cite{laurentloubesmarteau11} and has been previously evoked by \cite{butuceamatiaspouet09}. 
In the case of the convolution model on the real line, Laurent, Loubes and Marteau prove that if a test procedure is minimax for 
testing problem : $H_0^{D}: f_\varepsilon* f=f_\varepsilon* f_0$ versus $H_1^{D}:f_\varepsilon* (f-f_0)\in \mathcal{F}_D$ where
$$\mathcal{F}_D=\{g \text{ with smoothness } s' \text{ and } \|g\|^2\geq C'n^{-4s'/(4s'+1)}, \text{ with }s'=s+\nu\}, $$
then it is minimax for $H_0^{I}:  f= f_0$ versus $H_1^{I}: f-f_0\in \mathcal{F}_I$ where 
$$\mathcal{F}_I=\{f \text{ with smoothness } s \text{ and } \|f\|^2\geq Cn^{-4s/(4s+4\nu+1)}\}$$ but the reverse is not true
(here $n$ is the number of data and $\nu$ the smoothness index of the noise). 
This interesting conclusion (that we can conjecture true in our context also) does not make it any the less necessary to study the inverse problem here. Indeed, until the present work, the minimax rates were not known in the context of noisy spherical data. Moreover, when dealing with adaptive procedures, the link between the direct and inverse problems is not established yet.

In the present paper, the whole difficulty actually lies in the spherical geometry which complicates every steps that one encounters on $\R$.  Indeed, the efficient test statistic of  \cite{butucea07} was built upon a deconvolution kernel estimator of the quadratic functional $\int(f-f_0)^2$. It is well-known that such an estimator is  closely linked to the Fourier transform on $\R$. There exist kernel methods to treat density estimation for spherical data but  only for direct observations (see \cite{hallwatsoncabrera87}, and \cite{BRZ88}). Here in the spherical convolution context, Fourier analysis has a different behaviour and we resort to existing procedures to estimate the quadratic risk $\int(f-f_0)^2$. Those procedures (see \cite{kimkoo02}) are based on Fourier series which come down to projections. Consequently, the approach proves to be quite different than the one on the real line. The difficulty of testing in a spherical deconvolution model can be seen in the following way. If you use an orthogonal 
basis $(\psi_k)$ to estimate the unknown function $f$, then using U-statistics requires that the ``deconvolved`` basis $\phi_k$ (s.t. $\psi_k=f_\varepsilon*\phi_k$) is also (almost) orthogonal, which is delicate to realize. 
Thus one has to circumvent new problems linked to estimation of the quadratic functional, Fourier series, spherical context and convolution model setting.
This explains why we choose to use spherical harmonics and their good properties in terms of orthogonality. 

In this work, we establish several results for both smooth and supersmooth noises. We exhibit the optimal rates for non adaptive and adaptive cases.  We would like to stress that, in general, adaptivity implies a small loss in the minimax rate. Here, it is actually the case and our adaptive lower bound  proves that this loss is the least possible.  Furthermore we prove that our statistical procedure exactly attains these optimal rates. To complete these theoretical results, we implemented simulations to compare the performances of our procedure with other well-known directional testing procedures and applied our method to real data in paleomagnetism and UHECR incoming directions.

The plan of the paper is as follows. In Section 2, we give a brief overview about harmonic analysis on SO(3) and $\mathbb{S}^2$ which will be necessary throughout the paper. 
In Section 3 we define the test hypotheses and the smoothness assumptions about the unknown density $f$ and the noise $\varepsilon_i$. We also introduce the adaptive goodness-of-fit testing procedure. In Sections 4 and 5 we compute lower and upper bounds for testing rates for the ordinary smooth noise case.  The super smooth noise case is treated in Section 6. Finally, we give a simulation study  and applications on real data in Section 7. The proofs of the results are detailed in Section 8.

\section{Some preliminaries about harmonic analysis on SO(3) and \protect{$\mathbb S^2$}}

This part provides a brief overview of Fourier analysis on SO(3) and
$\mathbb{S}^2$. Most of the material can be found in expanded form
in \cite{HealyHendriksKim}, \cite{kimkoo02} \cite{vilenkin68}, \cite{talman68}, \cite{terras85} . 

Let $\mathbb{L}_2(\rm{SO(3)})$ denote the space of square integrable functions on SO(3), that
is, the set of measurable functions $f$ on SO(3) for which
$$\| f\|_2= \left (\int_{\rm{SO(3)}} |f(x)| ^2 dx \right )^{\frac 1 2} < \infty, $$
where $dx$ is the Haar measure on SO(3).

Let $D^l_{mn}$ for $-l\le
m,\;n\leq l,\;l=0,\;1,\ldots$ be the eigenfunctions of the Laplace
Beltrami operator on SO(3), hence, $\sqrt{2l+1} D^l_{mn},\; -l\le
m,\;n\leq l,\;l=0,\;1,\ldots$ is a complete orthonormal basis for
$\mathbb{L}_2({\rm SO(3)})$ with respect to the probability Haar measure.
Explicit formulae of the rotational harmonics $D^l_{mn}$ in terms of Euler angles exist but
we do not need it here. 
Next, for  $f \in \mathbb{L}_2({\rm SO(3)})$, we define the rotational
Fourier transform on SO(3) by the $(2l+1)\times (2l+1)$ matrices $f^{\star l}$ with entries
\begin{equation*}\label{rotationalFourierTransform}
 f_{mn}^{\star  l}=\int_{{\rm SO(3)}}f(g)D^l_{mn}(g)dg,
\end{equation*}
 where 
$dg$ is the probability Haar measure on SO(3).
 The rotational inversion can be obtained by
 \begin{align}\label{4}
f(g)&=\sum_{l\geq 0}\sum_{-l\le m,\;n\leq l} f_{mn}^{\star l}(2l+1){\overline {D^l_{mn}(g)}}.
\end{align}
(\ref{4}) is to be understood in  $\mathbb{L}_2$-sense although with
additional smoothness conditions, it can hold pointwise.

A parallel spherical Fourier analysis is available on
$\mathbb{S}^2$. Any point on $\mathbb{S}^2$ can be represented by
$$\omega=(\cos\phi\sin\theta,\sin\phi\sin\theta,\cos\theta)^t,$$
with  $\phi\in[0,2\pi),\;\theta\in [0,\pi)$. We also define the
functions:
\begin{equation*}
Y^l_m(\omega)=Y^l_m(\theta,\phi)=
\sqrt{\frac{(2l+1)}{4\pi }\frac{(l-m)!}{(l+m)!}}%
P^l_m(\cos \theta )e^{im\phi}, \label{5} 
\end{equation*}
for  $-l\le m\;\leq l,\;l=0,\;1,\ldots$,
$\phi\in[0,2\pi),\;\theta\in [0,\pi)$ and where $P^{l}_{m}$ are the associated Legendre functions. 
The functions $Y^{l}_{m}$ obey
\begin{equation}\label{harmonique_negative}
Y^{l}_{-m}(\theta,\phi)={(-1)}^m\overline{{Y}^{l}_{m}(\theta,\phi)}.
\end{equation}

Let $\mathbb{L}_2(\S^2)$ denote the space of square integrable functions on $\S^2$, that
is, the set of measurable functions $f$ on $\S^2$ for which
$$\| f\|_2= \left (\int_{\S^2} |f(x)| ^2 dx \right )^{\frac 1 2} < \infty, $$
where $dx$ is the Lebesgue measure on the sphere
$\mathbb{S}^2$.
It is well-known that $\mathbb{L}_2(\S^2)$ is a Hilbert space with the inner product
$$ \langle f,g\rangle_{\mathbb{L}_2}= \int_{\S^2} f(x)\overline{g(x)}dx, \quad f,g \in \L_2(\S^2).$$
The set $\{Y^l_m,\; -l\le m\;\leq l,\;l=0,\;1,\ldots\}$ is forming
an orthonormal basis of $\mathbb{L}_2(\mathbb{S}^2)$, generally
referred to as the spherical harmonic basis.
Again, as above, for  $f \in \mathbb{L}_2(\mathbb{S}^2)$, we define
the spherical Fourier transform on $\mathbb{S}^2$ by

\begin{equation}
f_{m}^{\star l}=\int_{\mathbb{S}^2}f(x){\overline{
Y^l_{m}(x)}}dx. \label{7}
\end{equation}
We think of (\ref{7}) as the vector entries of the
$(2l+1)$ vector 
$$ f^{\star l}=[ f_{m}^{\star l}]_{-l\le m \leq l},\;l=0,\;1,\ldots$$
 The spherical inversion can be obtained by
 \begin{align*}
f(\omega)=\sum_{l\geq 0}\sum_{-l\le m \leq l}f_{m}^{\star l}{Y^l_{m}(\omega)}.
\label{8}
\end{align*}
The bases detailed above are important because they realize a
singular value decomposition of the convolution operator created by
our model. In effect, we define for $f_\eps\in \mathbb{L}_2({\rm SO(3) }),
\; f\in \mathbb{L}_2(\mathbb{S}^2)$ the convolution by the following
formula:
\begin{equation}\label{convolution}
f_\eps*f(\omega)=\int_{{\rm SO(3)}}f_\eps(u)f(u^{-1}\omega)du
\end{equation}
and we have for all $-l\le m\;\leq l,\;l=0,\;1,\ldots$,
\begin{equation}
({f_\eps*f})_m^{\star l}=\sum_{n=-l}^l (f^{\star l}_{\eps})_{mn}
f^{\star l}_n=( f_\eps^{\star l} f^{\star l})_m .\label{convprod}
\end{equation}

We shall recall some basic facts which will be useful throughout the paper.
Let $\mathbb{H}_l$ the vector space spanned by $\{  Y^l_m= -l\leq m\leq l \}$ for each $l=0,1,\dots$.  Any element $h\in \mathbb{H}_l$  can be written as $h=\sum_{m=-l}^l h^{\star l}_{m} Y^l_m$ and thanks to Parseval equality we have $ \| h\|^2_2= \sum_{m=-l}^{l} | h^{\star l}_m|^2.$
Now according to (\ref{convprod}) we have 
$$ f^{\star l}_{\eps}: \mathbb{H}_l \rightarrow \mathbb{H}_l  \; \textrm{defined by} \; f^{\star l}_{\eps}h =\sum_{m=-l}^l\left (  \sum_{n=-l}^l (f^{\star l }_{\eps})_{ mn} h^{\star l}_n\right )Y^l_m.$$
We finally get the operator inequality
\begin{equation*}
\| f^{\star l}_{\eps} h\|_2 \leq \| f^{\star l}_{\eps}\|_{op} \| h\|_2, \quad \textrm{where}\quad  \| f^{\star l}_{\eps}\|_{op}= \sup_{h\neq 0, h\in \mathbb{H}_l} \frac{\| f^{\star l}_{\eps} h  \|_2}{\| h\|_2}.
\end{equation*}

\section{Model and assumptions}
We would like to present our results in terms of Sobolev classes (see e.g. \cite{HealyHendriksKim} for a definition on the sphere).
On the space $C^{\infty}(\S^2)$ of infinitely continuous differentiable functions on $\S^2$, consider the so-called Sobolev norm $\| \|_{W_s}$ of order $s$ defined in the following way. For any function $f=\sum_{lm} f^{\star l}_m Y^l_m$ let 
\begin{equation}\label{sobolevnorm}
 \| f\|_{W_s}^2= \sum_{l\geq 0} \sum_{m=-l}^l (1+l(l+1))^s | f^{\star l}_m|^2.
 \end{equation}
We denote by $W_s(\S^2)$ the vector space completion of $C^{\infty}(\S^2)$ with respect to the Sobolev norm (\ref{sobolevnorm}) of order $s$. 
Note that for any probability density $f$, 
$\| f\|_{W_s}^2\geq   | f^{\star 0}_0|^2=(4\pi)^{-1}$. 
Then, for some fixed constant $R>0$, let $W_s(\S^2,R)$ denote the smoothness class of densities $f\in W_s(\S^2)$ which satisfy
\begin{equation} \label{Sobolev}
\| f\|_{W_s}^2 \leq \frac1{4\pi}+R^2.
\end{equation}

For the uniform density of probability on the sphere namely $f_0={(4\pi)}^{-1}\1_\mathbf{\S^2}$, we want to test the hypothesis
$$ H_0: \quad f=f_0,$$
from observations $Z_1,\dots,Z_N$ given by model (\ref{model}). 
We consider the alternative
 $$H_1(s, R, \mathcal{C}\psi_N): f \in W_s(\S^2,R) \; \textrm{and}\;  \|f-f_0\|^2_2\geq \C\psi_N $$ 
where $\C$ is a constant and $\psi_N$ is the testing rate. 

We will say that the distribution of $\varepsilon$ is ordinary smooth of order $\nu$ if the rotational Fourier transform of $f_\eps$ satisfies the following assumption.

\begin{Hyp} For all $l\geq 0$, the matrix $f_\eps^{\star l}$ is invertible and there exist positive constants $d_0, d_1, \nu$ such that 
 $$\| f^{\star l}_{\varepsilon^{-1}}\|_{op} \leq d_0^{-1}l^\nu\quad\text{ and }\quad\| f^{\star l}_{\varepsilon} \|_{op} \leq d_1 l^{-\nu},$$
where we have denoted the matrix  $(f_\eps^{\star l})^{-1}$ by $f^{\star l}_{\eps^{-1}}$. 
\end{Hyp}

Recall that we assume that $f_\varepsilon$ is known, consequently $d_0$ and $\nu$ are also considered known. 
Some examples satisfying this assumption are given in Section~\ref{simus}.

In order to build a test statistic, as usual, we first have  to construct an unbiased estimator of the quadratic functional $\int_{\S^2} (f-f_0)^2= \| f-f_0\|^2_2$. To do so, we  
remark that thanks to Parseval equality: 
$$ \int_{\S^2} (f-f_0)^2= \sum_{l\geq 0} \sum_{m=-l}^l  | f^{\star l}_m - {f_0}^{\star l}_m |^2=   \sum_{l\geq 1} \sum_{m=-l}^l  | f^{\star l}_m|^2,$$
the last equality coming from the fact that $({f_0})^{\star l}_{m} \neq 0$ only for $(l,m)=(0,0)$. 
Since $ f^{\star l}=f^{\star l}_{\eps^{-1}} f_Z^{\star l}$ for $l=0,1,\dots$, we can write under Assumption 1
$$f^{\star l}_m= \sum_{n=-l}^l  (f^{\star l}_{\eps^{-1}})_{mn} (f_Z^{\star l})_n.$$
A natural estimator of $f^{\star l}_m$ is given by
$$  \hat{f}^{\star l}_m=\frac 1N \sum_{i=1}^N \sum_{n=-l}^l   (f^{\star l}_{\eps^{-1}})_{mn}  \overline{Y^l_n(Z_i)}.$$
If we denote by $\Phi_{lm}(x)=\sum_{n=-l}^l (f^{\star l}_{\eps^{-1}})_{mn} \overline{Y_n^l}(x)$
then 
$$ \hat{f}^{\star l}_m = \frac 1 N \sum_{i=1}^N \Phi_{lm} (Z_i).$$
Consequently, we can derive an unbiased estimator $T_{lm}$ of $|f^{\star l}_m |^2$
$$ T_{lm}= \frac{2}{N(N-1)}\sum_{i_1<i_2}\Phi_{lm}(Z_{i_1})\overline{\Phi_{lm}(Z_{i_2})},$$
and finally an estimator of $\| f-f_0\|^2_2$
\begin{equation*}\label{teststat}
T_L=\sum_{l=1}^L\sum_{m=-l}^l\frac{2}{N(N-1)}\sum_{i_1<i_2}\Phi_{lm}(Z_{i_1})\overline{\Phi_{lm}(Z_{i_2})}.
\end{equation*}

We can now define a test procedure 
$$\Delta=\begin{cases}
          1 \quad \text{ if }|T_L|>t^2\\
          0 \quad \text{ otherwise}
         \end{cases}$$
for a threshold $t^2$ to be suitably chosen. The choice of $L$ is crucial too, and this point will be solved in Sections 5 \& 6.\\

As one may have noticed, the noise smoothness hypothesis and hence the test procedure only rely on the Fourier transform of the noise density $f_\varepsilon$. Consequently, we do not need the existence of the density $f_\varepsilon$ but only the existence of the characteristic function $\E(D^{l}_{mn}(\varepsilon))$ of the variable $\varepsilon$.

\section{Lower bound for testing rate}

It is known that the rate of separation in the case of direct observations in dimension two is $N^{-{4s}/{4s+2}}$ when one considers for the alternative functions belonging to Sobolev ellipsoid in dimension 2
 (see \cite{IngsterSapatinas2009}). Let us see how it is
modified by the presence of a noise with smoothness $\nu.$

\begin{Th}\label{bi}
 Let  $s\geq 1$ and $\psi_N=N^{-2s/(2s+2\nu+1)}$.  Let $\eta\in (0,1) $. 
If ${\cal C}\leq K R^2$ where $K$ is a constant only depending on $d_0,d_1,\nu,s,\eta$, then 
$$\liminf_{N\to\infty}\inf_{\Delta_N}\left\{\P_{f_0}(\Delta_N=1)+\sup_{f\in H_1(s,R,{\cal C}\psi_N)}\P_f(\Delta_N=0)\right\}\geq  \eta$$
where the infimum is taken over all test procedures $\Delta_N$  based on the observations 
$ Z_1,\dots, Z_N$. 
\end{Th}

This means that testing with a faster rate than $\psi_N=N^{-2s/(2s+2\nu+1)}$ is impossible. If the distance beetween $f_0$ and the alternative is smaller than $\psi_N=N^{-2s/(2s+2\nu+1)}$, the sum of the error of the two kinds is close to 1. Nevertheless, it requires the knowledge of the smoothness index $s$. That is why we want to build on a so-called adaptive test procedure which does not depend on $s$. But we prove in the next statement that we have to face a phenomenon of ``lack of adaptability'' for our problem, i.e. it is not possible to test adaptively with the same rate. Indeed, in the context where $s$ is unknown and belongs to some set $\mathcal{S}$, there is not any universal test with small error for each $s\in\mathcal{S}$. The price to pay for adaptivity is an extra factor $\sqrt{\log\log N}$ in the separation rate.

\begin{Th}\label{biadap}
 For all $s\geq 1$, let  $\psi_N^{ad}(s)=(N/\sqrt{\log\log(N)})^{-2s/(2s+2\nu+1)}$. 
Let $\mathcal{S}$ be a set such that $\mathcal{S}\cap[1,\infty)$ contains an interval. If ${\cal C}\leq K R^2$ where $K$ is a constant only depending on $d_0,d_1,\nu,\mathcal S$, then, 
$$\liminf_{N\to\infty}\inf_{\Delta_N}\left\{\P_{f_0}(\Delta_N=1)+\sup_{s\in\mathcal{S}}\sup_{f\in H_1(s,R,  {\cal C}\psi_N^{ad}(s))}\P_f(\Delta_N=0)\right\}\geq 1$$
where the infimum is taken over all test procedures $\Delta_N$  based on the observations $ Z_1,\dots, Z_N$. 
Moreover any rate faster than $\psi_N^{ad}$ will also lead to a lower-bounded error. 
\end{Th}

\section{Upper bound for testing rate}

In order to construct an adaptive procedure of testing, we shall use the following exponential inequality. 

\begin{Lemma} \label{inegexp} There exist $K_0, K_1$ such that, for all sequence $u_N$,
$$\P_0(|T_L|>L^{2\nu+1}u_N/N)\leq K_1\exp(-K_0u_N^2)$$
provided that $u_NL^{-1}$, $LN^{-2}u_N^{8}$ and $LN^{-1}u_N^{3}$ are bounded.
\end{Lemma}

Actually the term $L^{2\nu+1}/N$  is the order of the variance of $T_L$ under $H_0$.
We denote $\lceil x\rceil$ the smallest integer larger than or equal to $x.$

\begin{Th}\label{bs}
 Assume $s\geq 1$ and $\psi_N^{ad}=(N/\sqrt{\log\log N})^{-2s/(2s+2\nu+1)}$.  We consider the set 
 $\mathcal{L}=\{2^{j_0},\dots, 2^{j_m}\}$ where 
 $j_0=\lceil \log_2(\log\log N)\rceil$, $j_m=\lceil \log_2(N(\log\log N)^{-3/2})\rceil$ and the adaptive test statistic
 $$D_N=\1_{\{\max_{L\in\mathcal{L}} (|T_L|/t_L^2)>\sqrt{2/K_0}\}}$$
 with $t_L^2=L^{2\nu+1}\sqrt{\log\log N}/N.$ 
Then, if $\C>\sqrt{2K_0^{-1}}+((4\pi)^{-1}+R^2)2^{2s}$,
$$\lim_{N\to\infty}\left\{\P_{f_0}(D_N=1)+\sup_{f\in H_1(s,R,\C\psi_N^{ad})}\P_f(D_N=0)\right\}=0.$$
\end{Th}

This result shows that our procedure achieves the minimax rate of testing, and the limiting distribution of the asymptotically minimax test statistic is degenerate.

Note that the direct case (without noise) is included in this result, taking $\varepsilon=Id$, $f^{\star l}_\varepsilon=Id$, $\nu=0$. In this case, the separation rate is $(N/\sqrt{\log\log N})^{-2s/(2s+1)}$. To our knowledge, even in this simpler case, this result was not established yet.

\section{Super smooth noise}

In this section, we deal with the case of a super smooth noise. This kind of noise is of interest since it includes the Gaussian distribution. 
We will say that the distribution of $\varepsilon$ is super smooth of order $\nu$ if the rotational Fourier transform of $f_\eps$ satisfies
\begin{Hyp} 
For all $l\geq 0$, the matrix $f_\eps^{\star l}$ is invertible and there exist reals $\nu_1\leq\nu_0$,  and positive constants $d_0, d_1, \delta,\beta$ such that 
$$\| f^{\star l}_{\varepsilon^{-1}}\|_{op}\leq d_0^{-1}l^{-\nu_0}\exp(l^\beta/\delta)\quad
\text{ and }\quad\| f^{\star l}_{\varepsilon}\|_{op}\leq d_1l^{\nu_1}\exp(-l^\beta/\delta).$$
 \end{Hyp}

In this case, we present a similar test statistic but with a different threshold $t_L$. Moreover it is sufficient to consider only one $L^*$ instead of a maximum.

\begin{Th}\label{bs2} 
 Let $\psi_N=(\log N)^{-2s/\beta}$ and $K_0>0$. We consider $L^*=\left\lfloor\left({\delta}\log(N)/8 \right)^{1/\beta}\right\rfloor$ and  the test statistic
 $$D_N=\1_{\{ |T_{L^*}|/t_{L^*}^2>K_0\}}$$
 with $t_L^2=L^{-2\nu_0+1}\exp(2L^\beta/\delta)/N.$ 
Then, if $\C>K_0+((4\pi)^{-1}+R^2)(\delta/16)^{-2s/\beta}$,
$$\lim_{N\to\infty}\left\{\P_{f_0}(D_N=1)+\sup_{f\in H_1(s,R,\C\psi_n)}\P_f(D_N=0)\right\}=0.$$
\end{Th}

We observe that in this case the separation rate is very slow $\psi_N=(\log N)^{-2s/\beta}$. However, this rate is reached without any knowledge on the smoothness of $f$. Moreover, we prove that this is the optimal rate:

\begin{Th}\label{bi2}
  Let  $s\geq 1/2$ and $\psi_N=(\log N)^{-2s/\beta}$. 
If ${\cal C}\leq K R^2$ where $K$ is a constant only depending on $d_0,d_1,\nu_0,\beta,\delta, s$, then 
$$\liminf_{N\to\infty}\inf_{\Delta_N}\left\{\P_{f_0}(\Delta_N=1)+\sup_{f\in H_1(s,R,{\cal C}\psi_N(s))}\P_f(\Delta_N=0)\right\}\geq  1$$
where the infimum is taken over all test procedures $\Delta_N$  based on the observations 
$ Z_1,\dots, Z_N$. 
\end{Th}

The deterioration of the rate in the case of a super smooth noise is a well-known phenomenon in convolution models (see e.g. \cite{fan91}).

\section{Numerical illustrations}\label{simus}

In this section, we highlight the numerical performances of our test procedure and compare them with other well-known directional test procedures. We both deal with simulations and real data in astrophysics and paleomagnetism. 

\subsection{The testing procedures}
We shall now explain the various directional testing procedures we consider in this numerical section. 

Let us start with our adaptive testing procedure that will be denoted by SHT (as Spherical Harmonics Test). The test is described in Theorems~\ref{bs}  and \ref{bs2}. For the quantile $K_{0}$, we generate 1000 times $N$ observations uniformly under $H_0$. 
Then, we compute by 1000 Monte Carlo runs the $5\%$ quantile of the statistics $\max_{L \in \mathcal{L}}(|T_L|/t^2_L)$ defined in the theorems. In the ordinary smooth case, we use $j_{m}=\lceil\log_{2}(  N^{1/3}(\log\log N)^{-3/2}  )\rceil$. Indeed, Theorem~\ref{bs} is valid for a large class of $j_m$ (see the proofs) and for the implementation we choose a small one for the sake of efficiency.
We point out that our numerical procedure is notably fast all the more so as we are in dimension 2. Furthermore, we do not have any tuning parameter. 

To compare our results, we have implemented two other procedures. 

The first one is called \textit{the Nearest Neighbour test} and was proposed by \cite{QuashnockLamb93}. It will be denoted NN in the sequel.  For each observation $Z_i$, one must compute the distance $Y_i$ to its nearest neighbour. The Wilcoxon test statistic is
 $$W=\sqrt{12N}\left( \frac 1 2 -\frac 1 N \sum_{i=1}^{N} \phi(Y_i) \right ),$$
 where $\phi(z)= 1-[(1+\cos z)/2]^{N-1}$. The distribution of $W$ is asymptotically standard Gaussian. Notice that this test was designed for non-noisy data.
 
The second procedure was introduced by \cite{Beran68} and \cite{Gine75}. The test statistic is 
$$ F_{N} = \frac{3N}2-\frac4{N\pi}\sum_{i=1}^{N-1}\sum_{j=i+1}^{N}d(Z_{i},Z_{j})+\sin(d(Z_{i},Z_{j}))$$ 
where $d(Z_{i},Z_{j})=\arccos\langle Z_{i}, Z_{j}\rangle$ is the spherical distance between $Z_{i}$ and $Z_{j}$, and the quantiles are computed via simulations under $H_{0}$.
Again, this test was designed for non-noisy data.

In \cite{FDKP12}, the authors implement two procedures called Multiple and PlugIn for the noise free case. These procedures are both based on needlets which can be seen as the wavelets on the sphere. The Multiple test is based on a family of linear estimators of the density $f_Z$ whereas the PlugIn test  considers a hard thresholding procedure on needlets.

\subsection{Generating the noise}

In our theoretical statements, we talked about ordinary and super smooth noises on the group SO(3), but what does it mean in practice? In fact, there exist concrete examples of random matrices which could be generated according to densities which meet those smoothness assumptions. We will particularly highlight two cases, the Rotational Laplace and the Gaussian densities on SO(3) (for further details see  \cite{kimkoo02}). To the best of our knowledge, they have never been implemented in practice.  The first one is an ordinary smooth density and the second one a super smooth one. 

As explained in Section 3, the noise smoothness can be characterized by the decay of its rotational Fourier transform. 
The Rotational Laplace distribution is the rotational analogue of the well-known Euclidean Laplace distribution (known also as double exponential distribution). It has been discussed in depth in \cite{HealyHendriksKim}. Its expanded form in terms of rotational harmonics is the following
\begin{equation}  \label{laplace} f_\varepsilon= \sum_{l\geq 0} \sum_{m=-l}^l (1+ \sigma^2l(l+1))^{-1} (2l+1)\overline{D^l_{mm}},  \end{equation} 
for some $\sigma^2>0$ which is a variance parameter. Hence we have
$$ (f^{\star l}_{\varepsilon})_{ mn} = (1+\sigma^2l(l+1))^{-1}\delta_{mn},$$
for $l=0,1,\dots$ and where $\delta_{mn}=1$ if $m=n$ and is $0$ otherwise. The Laplace distribution is ordinary smooth with a smoothness index $\nu=2$. 

Let us present now the Gaussian distribution. The distribution can be written as follows (see \cite{kimkoo02}) 
\begin{equation}  \label{gaussian}  f_\varepsilon = \sum_{l\geq 0} \sum_{m=-l}^l \exp(-\sigma^2l(l+1)/2)(2l+1) \overline{D^l_{mm}},  \end{equation}
for $\sigma>0$.
This is an example of a super smooth distribution with $\delta=2/\sigma^2$ and $\beta=2$ following the terminology in Section 6. 

We would also like to make a remark about how to generate random matrices according to the Laplace or the Gaussian distribution. After rewriting carefully their density expressions in terms of rotational harmonics given by (\ref{laplace}) and (\ref{gaussian}), it turned out that $f_\varepsilon(u)$ only depends on the 
angle of the rotation $u$, say $\theta$. Then the simulation of a rotation following $f_\varepsilon$ amounts to pick at random an axis and perform a rotation about this axis by an angle following the law $f_\varepsilon(\theta)(1-\cos(\theta))/\pi$.

\subsection{Alternatives}
We have investigated the performances of the various testing procedures described above for two kind of alternatives. These alternatives aim at describing different relevant scenarios in practice. 

The first family of alternatives is non isotropic, unimodal with a Gaussian shape. More precisely, it is a mixture of a Gaussian-like density with the uniform density $f_0$.  We will denote this alternative by $H_{1}^a$. The $H_{1}^a$ density has the following form
$$
f(x)=(1-\delta)f_0+\delta h_{\gamma}(x),
$$
where  $h_{\gamma}(x):=C_\gamma \exp(-{d(x, x_0)^2}/({2\gamma^2}))$, $d$ is the spherical distance,
$C_\gamma$ is a normalization constant such that $\int_{\mathbb{S}^2}f(x)dx=1$ and $x_0$ is $(\pi/ 2, 0)$ in spherical coordinates. In the sequel, we chose $\delta=0.08$ and $\gamma=5\pi/180$ i.e. $\gamma=5^{\textrm{o}}$.  Remark that with this choice of parameters, the dose of uniformness injected in $H_1^a$ is  high and complicates the detection of the alternative from the null hypothesis. \\
This density is particularly meaningful in the field of astrophysics since very often one seeks for some departure from isotropy and some principal direction. As \cite{FDKP12} also considered $H_1^a$, this will permit us to compare the performances of our test to the Multiple and PlugIn procedures of \cite{FDKP12}. 
Figure~\ref{visuH1a} allows to visualize this alternative. The density is represented in spherical coordinates as a surface $z=f(\theta,\phi)$. To visualize points on the sphere, we use Hammer projection, because of its equal-area property. 
 \newline \\

\begin{figure}[!h]\begin{center}
\begin{tabular}{ccc}
\includegraphics[scale=0.25]{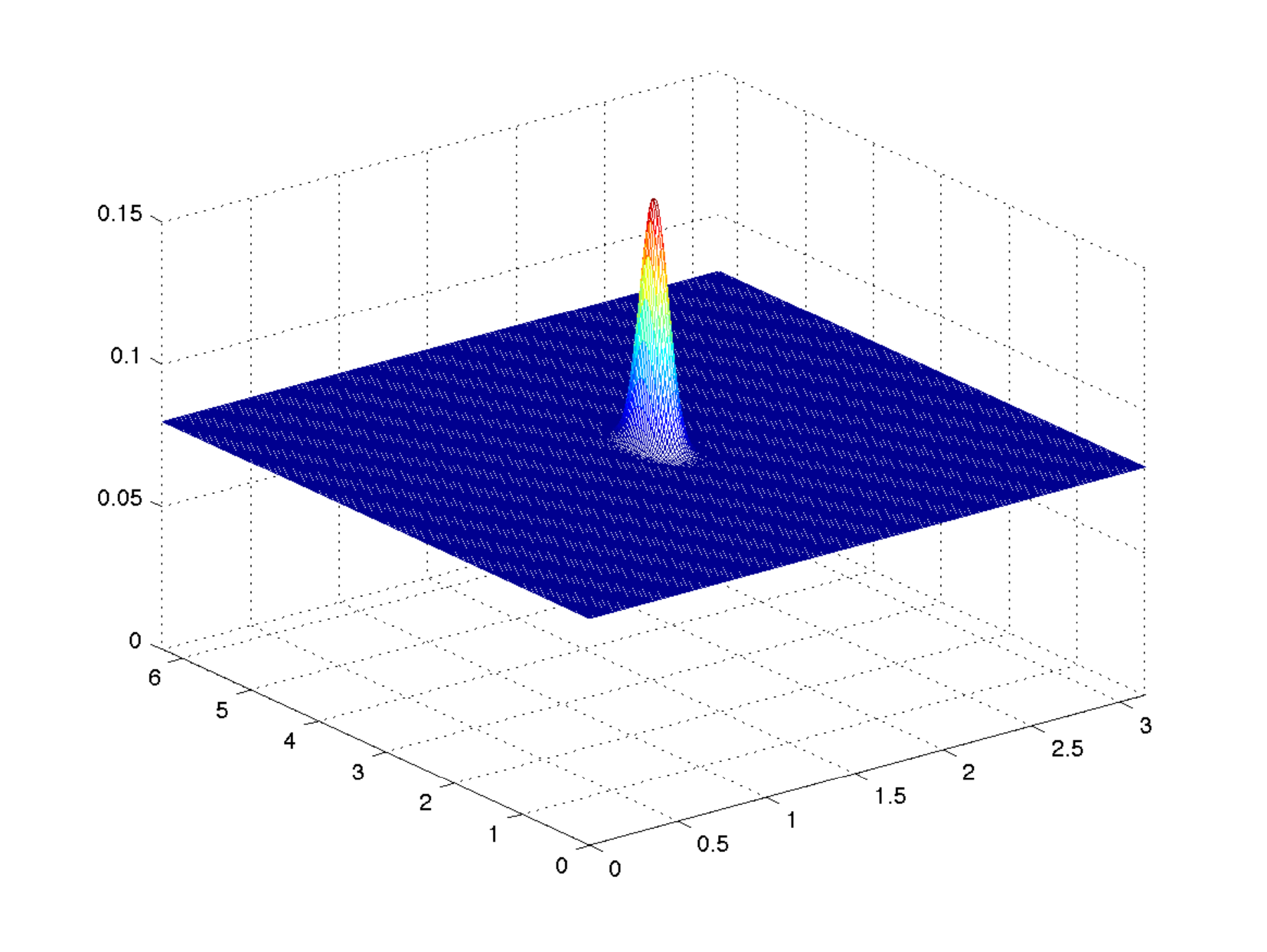}&
\includegraphics[scale=0.3]{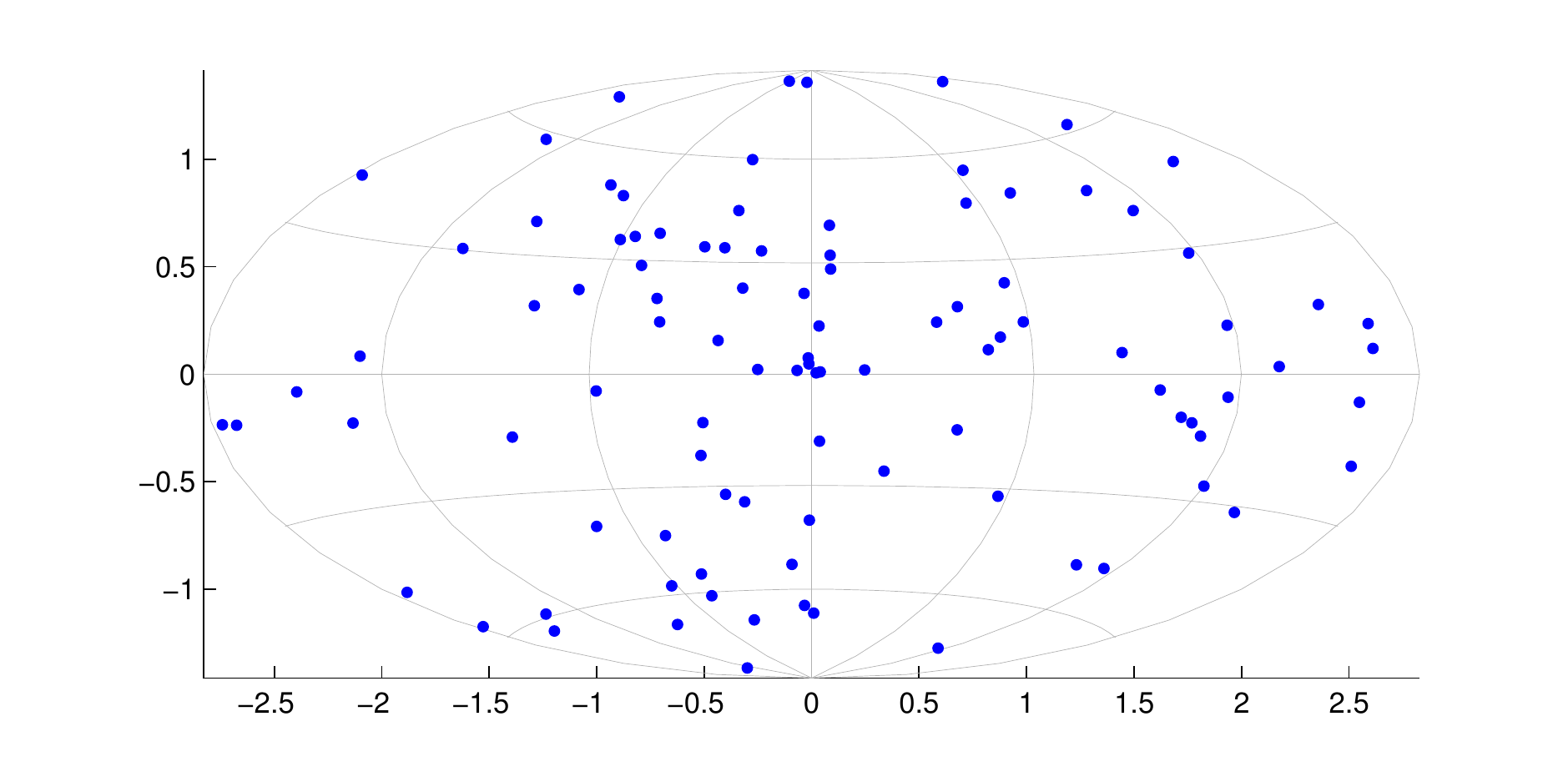}&
\includegraphics[scale=0.3]{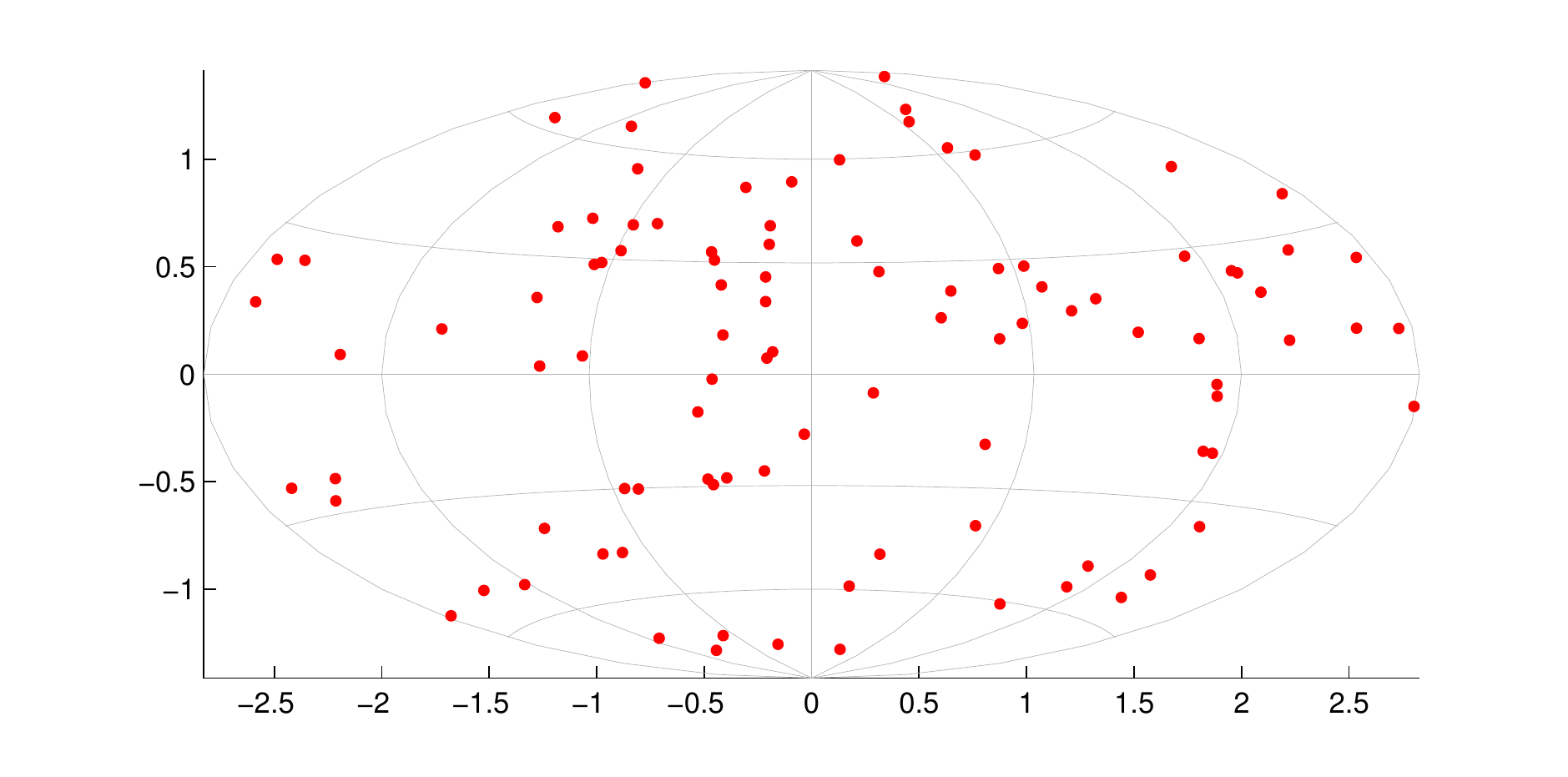}\\ 
a & b & c
\end{tabular}
\caption{a/ Representation of the $H_1^a$ density in spherical coordinates. b/ 100 random draws $X_{i}$ from $H_{1}^a$ distribution, 
c/ 100 random draws $Z_{i}$ from $H_{1}^a$ convolved with a Laplace noise with variance 0.1}
\label{visuH1a}
\end{center}\end{figure}

The second alternative that we consider and which is denoted by $H_{1}^b$ is the Watson distribution (\cite{watson65}).  Its density  is
$$f(\theta,\phi)=C\exp(-2\cos^{2}(\theta))$$
with $C$ such that $\int_{0}^{2\pi}\int_{0}^{\pi}f(\theta,\phi)\sin(\theta)d\theta d\phi =1$. 
This distribution has a girdle form, distributed around the equator. This choice is motivated by two reasons : first, this gives an alternative very different from $H_1^a$, second, it plays a role in applications. 
For example, in the case of gamma-ray bursts (see \cite{VedrenneAtteia09}), many theories assumed that the sources of these flashes 
were located around the galactic plane (then a girdle distribution), whereas other proposed that gamma-ray bursts come from beyond the Milky Way (rather a uniform distribution).
Figure~\ref{visuH1d} presents this alternative. Notice that the presence of noise (Figure~\ref{visuH1d} c/) prevents from seeing the equatorial nature of the distribution.

\begin{figure}[!h]\begin{center}
\begin{tabular}{ccc}
\includegraphics[scale=0.25]{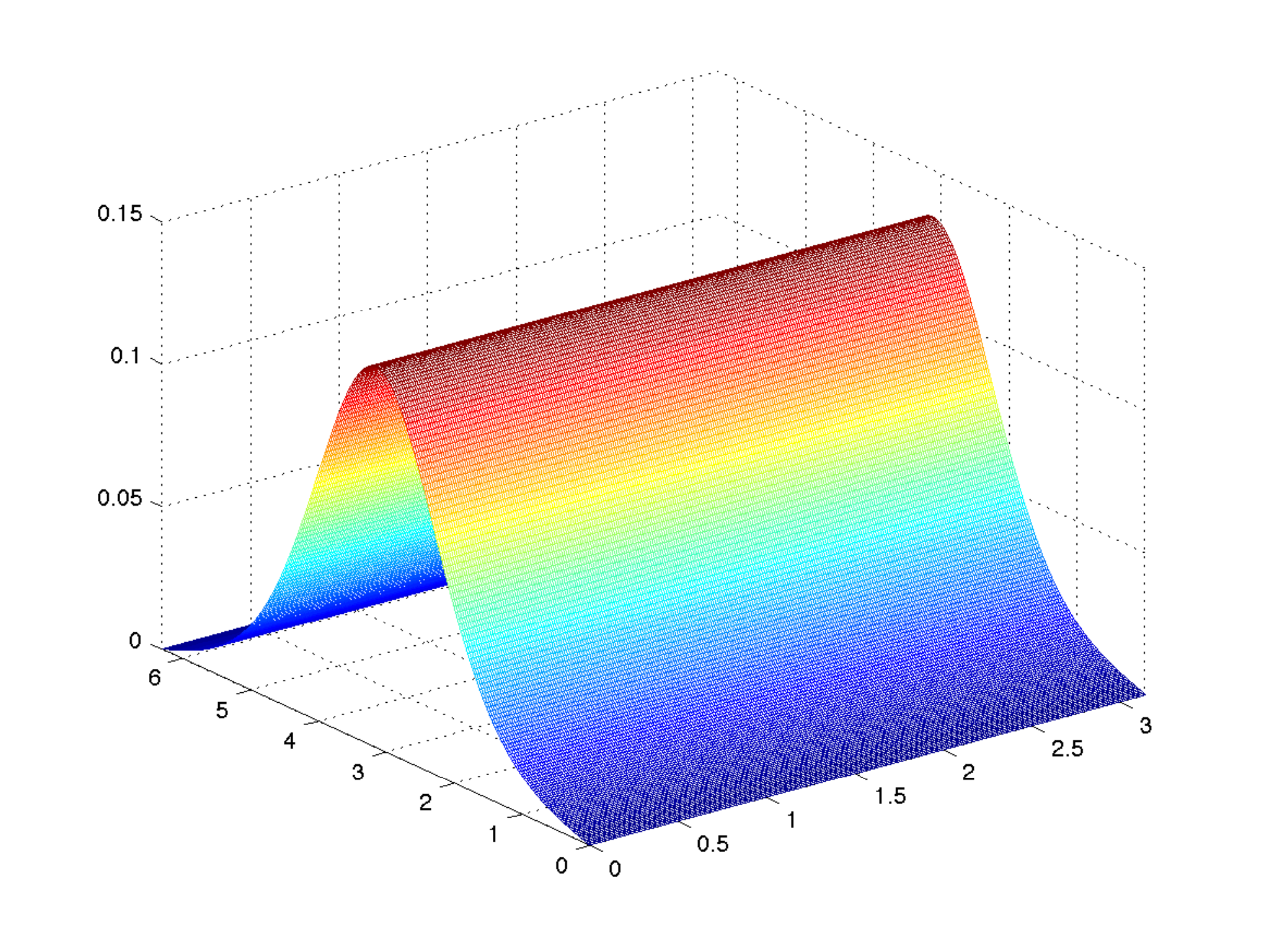}&
\includegraphics[scale=0.3]{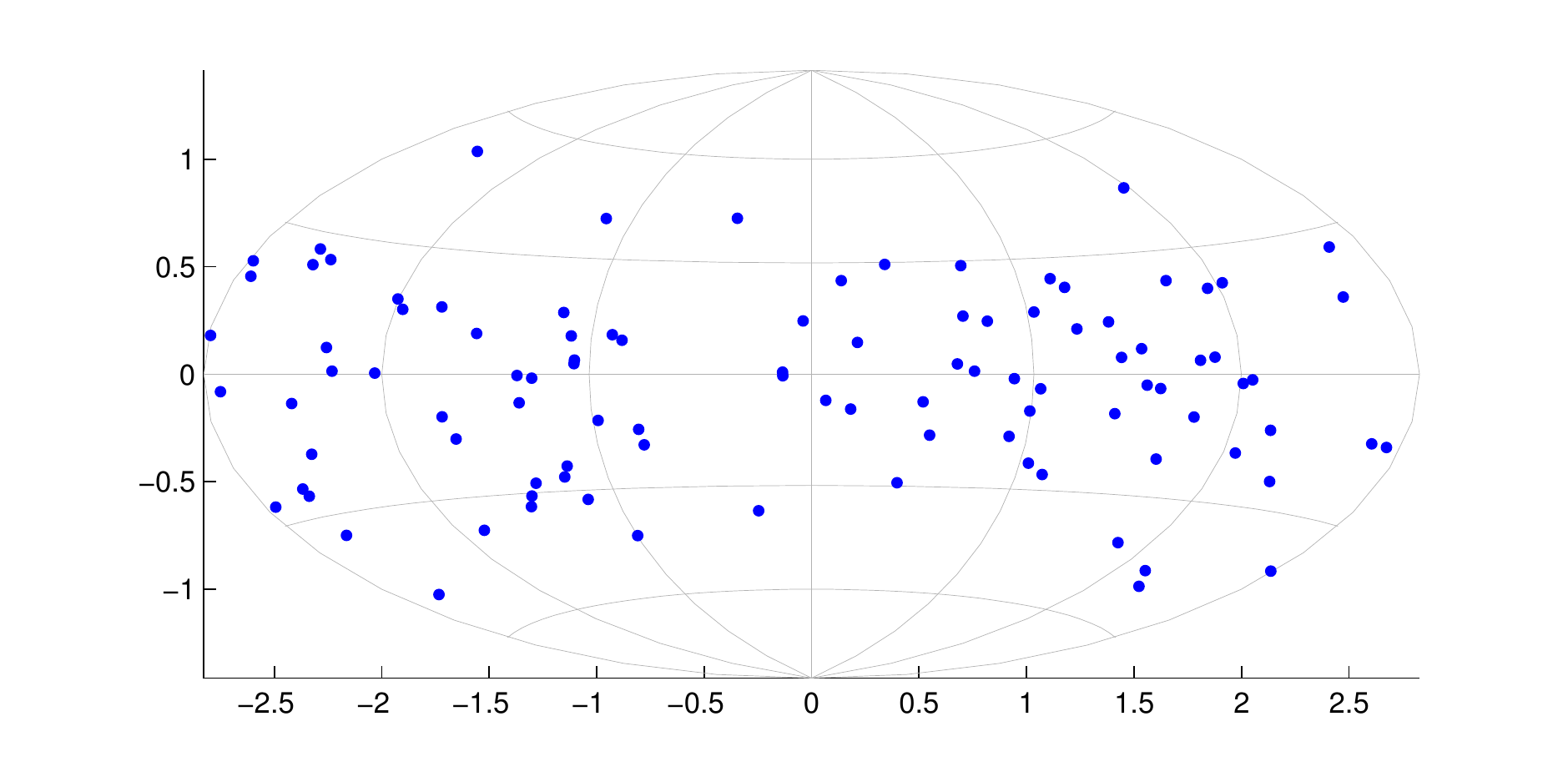}&
\includegraphics[scale=0.3]{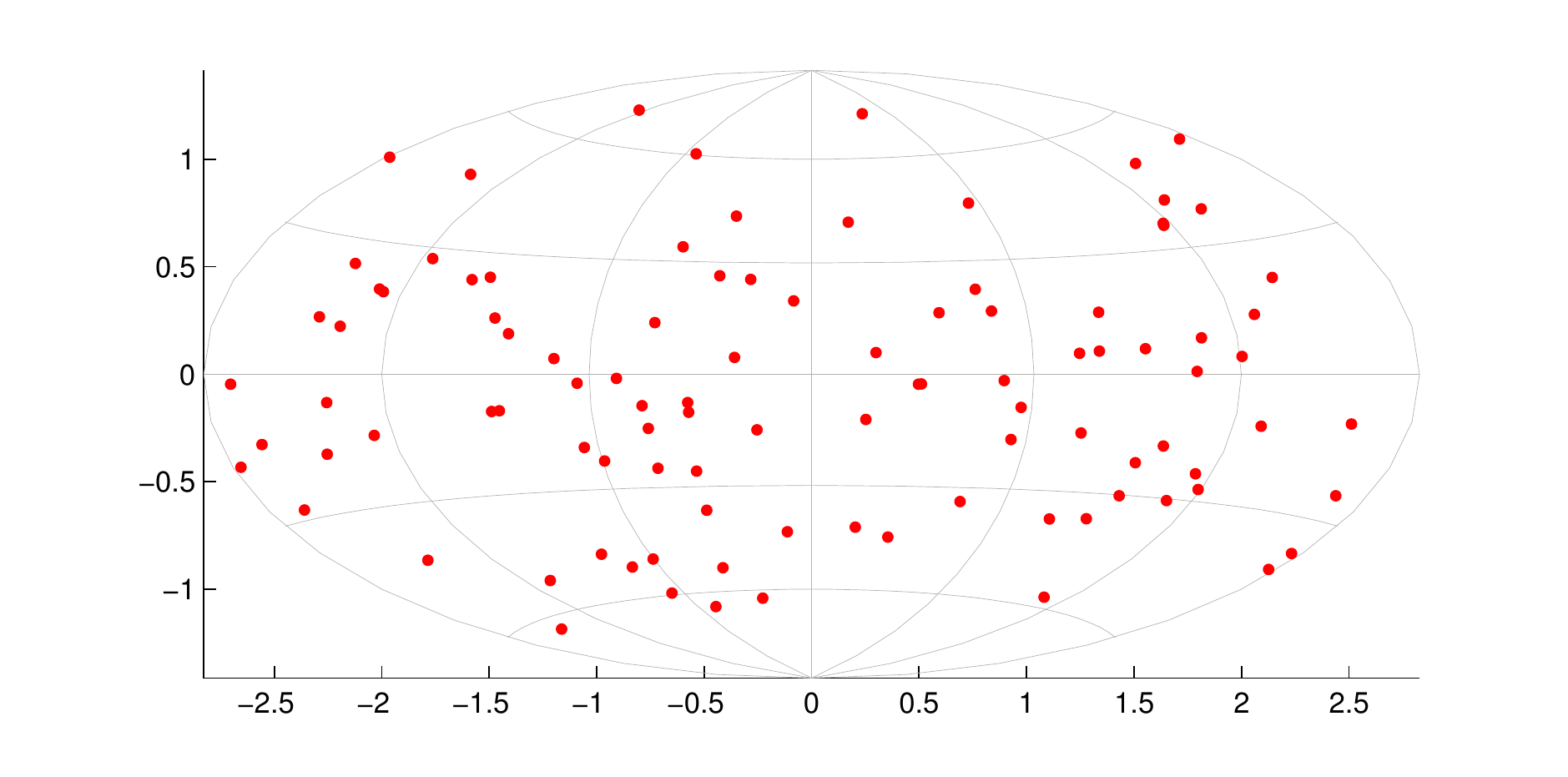}\\
a & b & c
\end{tabular}
\caption{a/ Representation of the Watson density in spherical coordinates. b/ 100 random draws $X_{i}$ from Watson distribution, 
c/ 100 random draws $Z_{i}$ from Watson distribution convolved with a Gaussian noise with variance 0.2}
\label{visuH1d}
\end{center}\end{figure}

\subsection{Simulations}

For the two alternatives, we computed the test power for a prescribed level of  5\%, for our test procedure (denoted by SHT), for the nearest neighbour test (denoted by NN) and the Beran-Giné test (denoted by BG).
Tables~\ref{H1aN100}-\ref{H1dN250}  give the results in percent for different kind of noisy data : no noise, Laplace noise with variance 0.05, 0.1, 0.2 and Gaussian noise with variance 0.05, 0.1, 0.2. \\

\begin{table}[!h]  \begin{center}
\begin{tabular}{c|c|c c c| c c c}
Noise type & No noise & \multicolumn{3}{c}{Laplace} & \multicolumn{3}{c}{Gaussian}\\
variance &&  0.05  &  0.1  &  0.2 &  0.05 &  0.1 &  0.2 \\
\hline
SHT  &53  & 18 & 13 & 10 & 13& 8& 8\\
NN & 19& 10& 10 & 8 &9 &7 & 7  \\ 
BG & 30&19 & 18& 10& 18& 12&11 \\
Multiple (J=3) & 62 &\\
PlugIn (J=3) & 63
\end{tabular}
\caption{Test powers for $N=100$ and  $H_{1}^a$}
\label{H1aN100}
\end{center}\end{table}

\begin{table}[!h]  \begin{center}
\begin{tabular}{c|c|c c c| c c c}
Noise type & No noise & \multicolumn{3}{c}{Laplace} & \multicolumn{3}{c}{Gaussian}\\
variance &&  0.05  &  0.1  &  0.2 &  0.05 &  0.1 &  0.2 \\
\hline
SHT  & 95  & 45 & 35 & 19 & 20 & 19& 12\\
NN & 26& 11& 11 & 8 &11 & 8 & 9  \\ 
BG & 61&43 & 38& 24& 41& 30&20
\end{tabular}
\caption{Test powers for $N=250$ and  $H_{1}^a$}
\label{H1aN250}
\end{center}\end{table}

\begin{table}[!h]\begin{center}
\begin{tabular}{c|c|c c c| c c c}
Noise type & No noise & \multicolumn{3}{c}{Laplace} & \multicolumn{3}{c}{Gaussian}\\
variance & &  0.05  &  0.1  &  0.2 &  0.05 &  0.1 &  0.2 \\
\hline
SHT & 100 & 98 & 83 & 49 & 45 & 31 & 21\\
NN &64 & 34 & 19 & 10 & 30 & 17 & 10  \\ 
BG & 93 & 48 & 27 & 15 & 32 & 19 & 10\\
\end{tabular}
\caption{Test powers for $N=100$ and  $H_{1}^b$}
\label{H1dN100}
\end{center}\end{table}

\begin{table}[!h]\begin{center}
\begin{tabular}{c|c|c c c| c c c}
Noise type & No noise & \multicolumn{3}{c}{Laplace} & \multicolumn{3}{c}{Gaussian}\\
variance & &  0.05  &  0.1  &  0.2 &  0.05 &  0.1 &  0.2 \\
\hline
SHT  & 100 & 100 & 100 & 95 & 73 & 86 & 51\\
NN & 91 & 52 & 34 & 16 & 46 & 26 & 10  \\ 
BG &  100 & 99 & 89 & 41 & 99 & 71 & 15
\end{tabular}
\caption{Test powers for $N=250$ and  $H_{1}^b$}
\label{H1dN250}
\end{center}\end{table}

As expected, the increase of the size sample improves the power, whereas presence of noise reduces it. 

For the $H_{1}^{a}$ alternative, in absence of noise, our procedure perform better than the two others. When adding some noise, our procedure has better results than the NN one but slightly worse than the BG procedure. 
It is only possible to compare our results with those of \cite{FDKP12} (see their Figure 7) for the noise free case and for $N=100$ observations. Their Multiple and PlugIn procedures for a resolution level equal to $3$ performs better than the three other procedures presented here. Nonetheless, it is important to notice that our procedure SHT 
is entirely data-driven and has no tuning parameter, contrary to the one of \cite{FDKP12}. In addition, \cite{FDKP12} has not dealt with the noise scenario and their procedures are rather complicated and lengthy for practical purposes.

When the alternative is the Watson distribution, our procedure performs pretty well and is clearly better than the others. Indeed the test powers for our procedure are often twice higher.
\bigskip

We also computed the ROC curves for the three methods for different noise and numbers of observations settings. Let us recall that the \textit{Receiver Operating Characteristic} curves allow to illustrate the performance of a test by plotting the true positive rate vs. the false positive rate, at various threshold settings. Roughly speaking, greater the area under the ROC curve, better the test. 

\begin{figure}[!h]\begin{center}
\begin{tabular}{cc}
\includegraphics[scale=0.3]{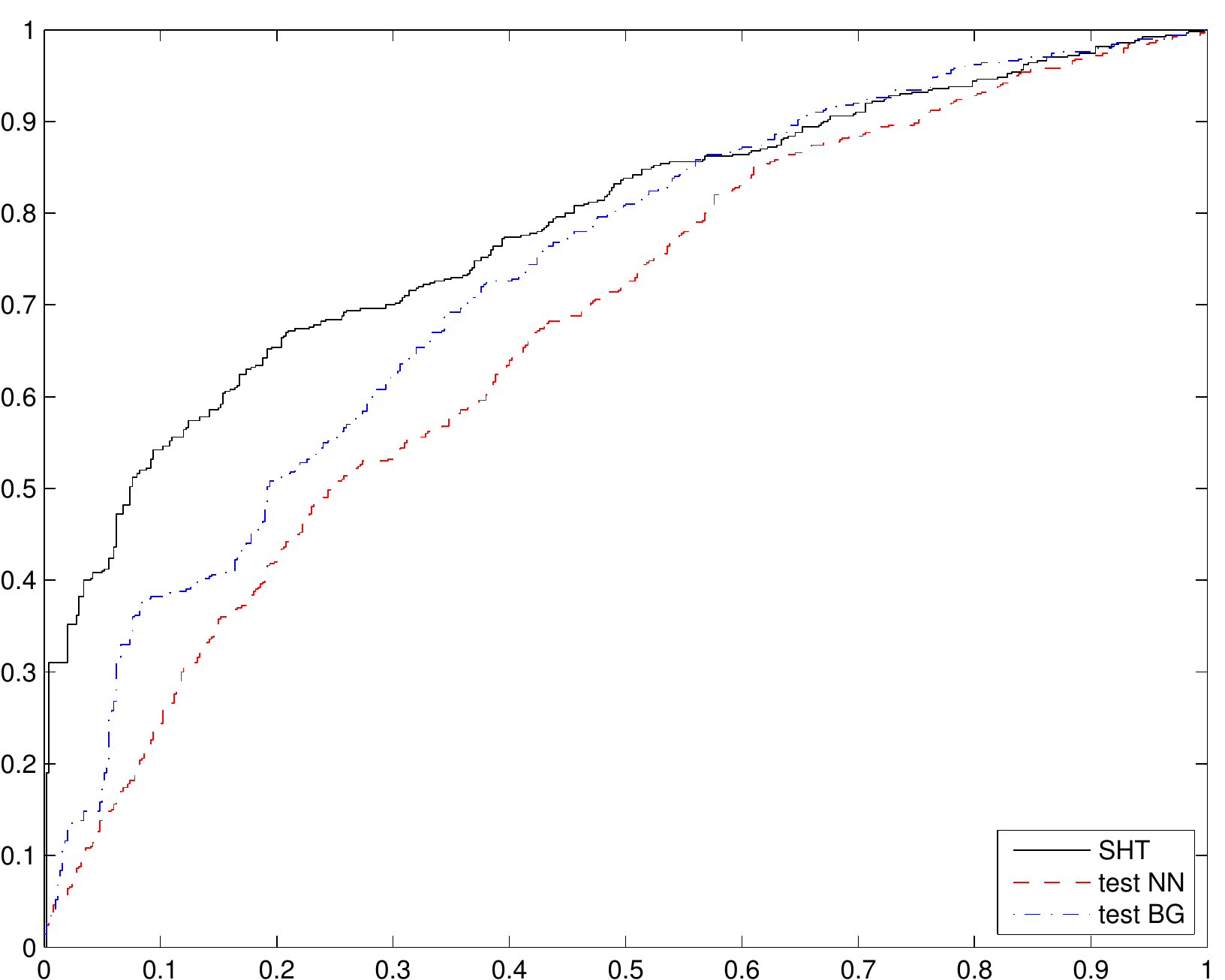}&
\includegraphics[scale=0.3]{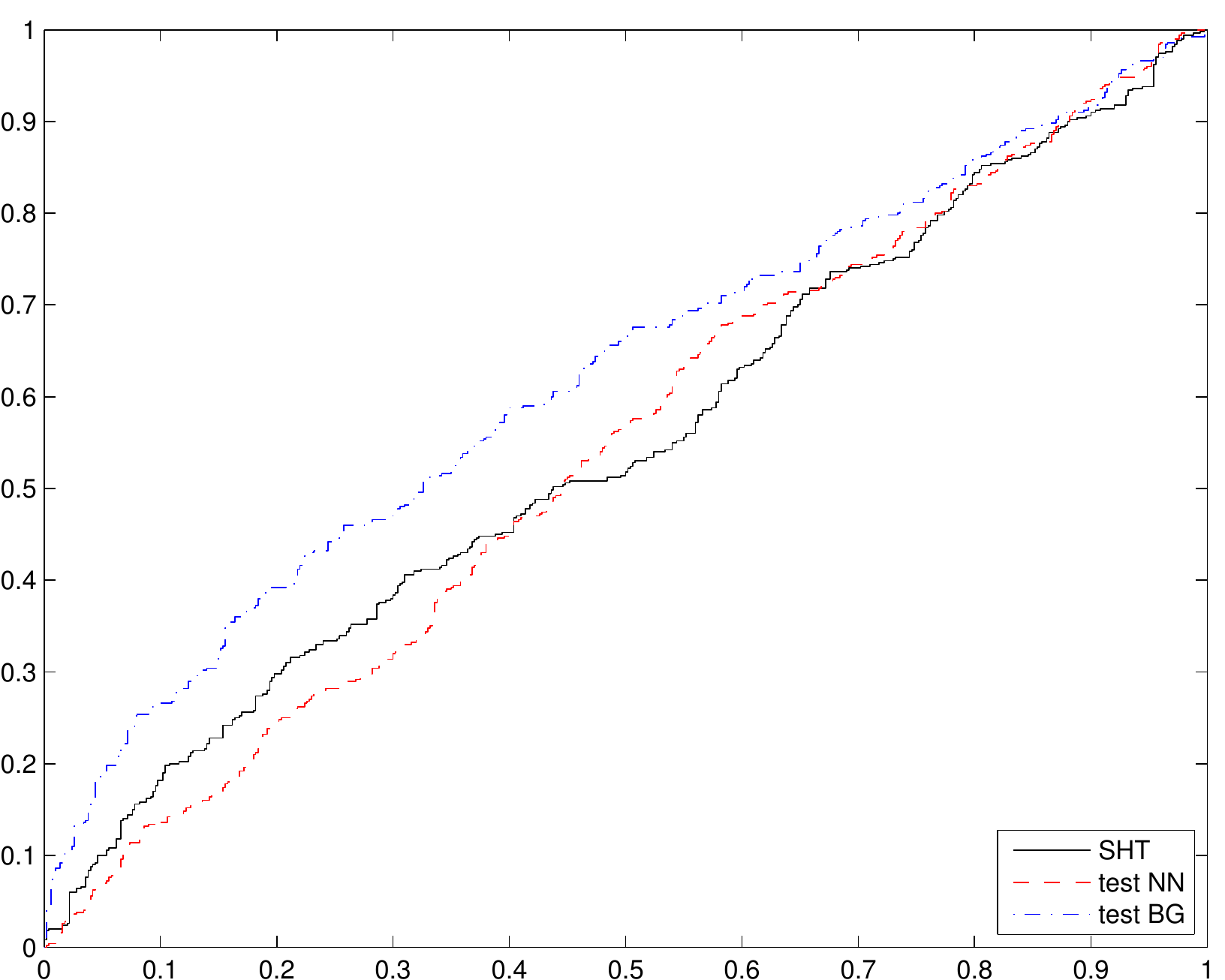}\\
a & b 
\end{tabular}
\caption{ROC Curves for the three methods and for the alternative $H_1^a$: a/ No noise and $N=100$. b/ Laplace noise with variance 0.1 and $N=100$.
}
 \label{ROCH1a}
\end{center}\end{figure}

\begin{figure}[!h]\begin{center}
\begin{tabular}{ccc}
\includegraphics[scale=0.3]{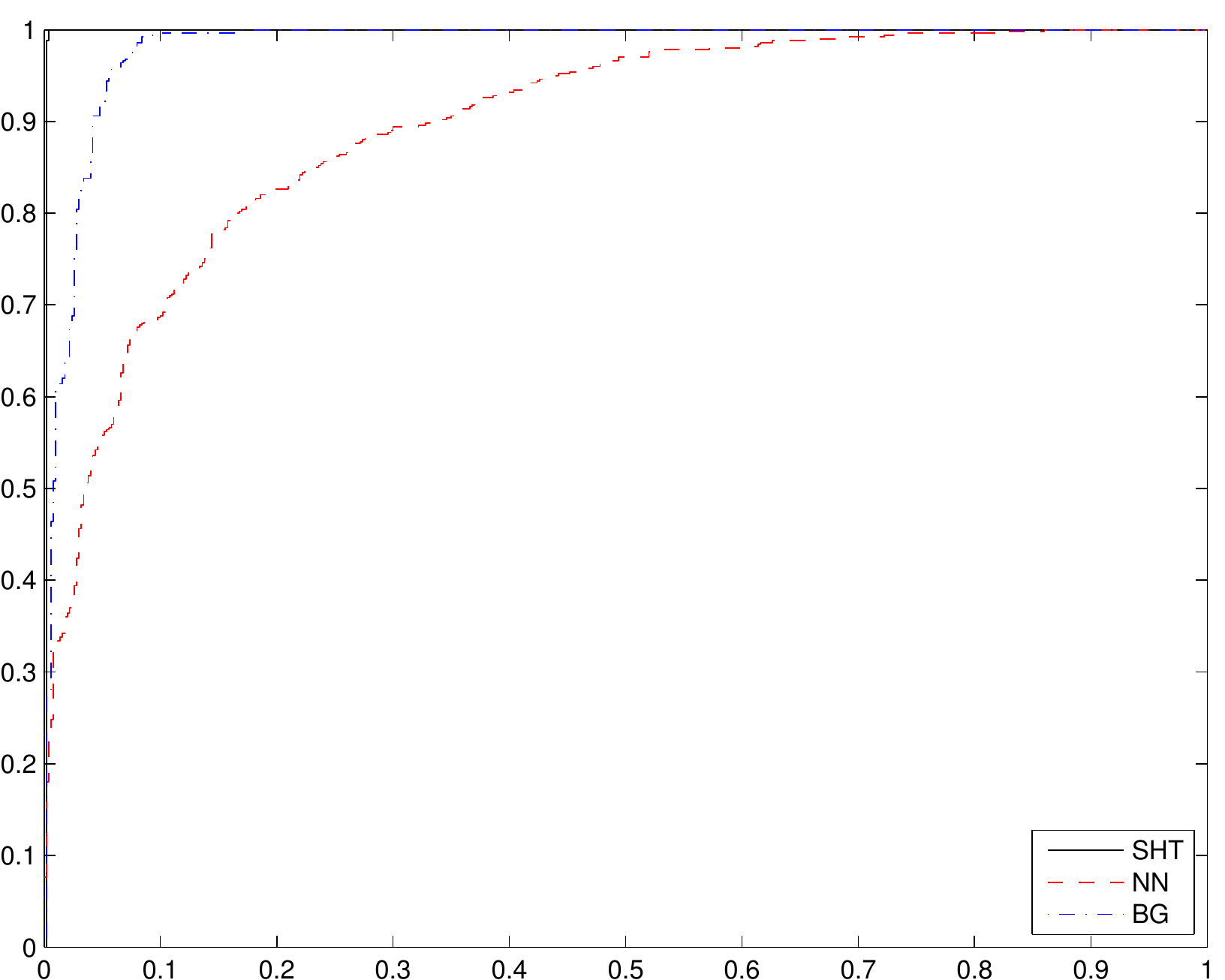}&
\includegraphics[scale=0.3]{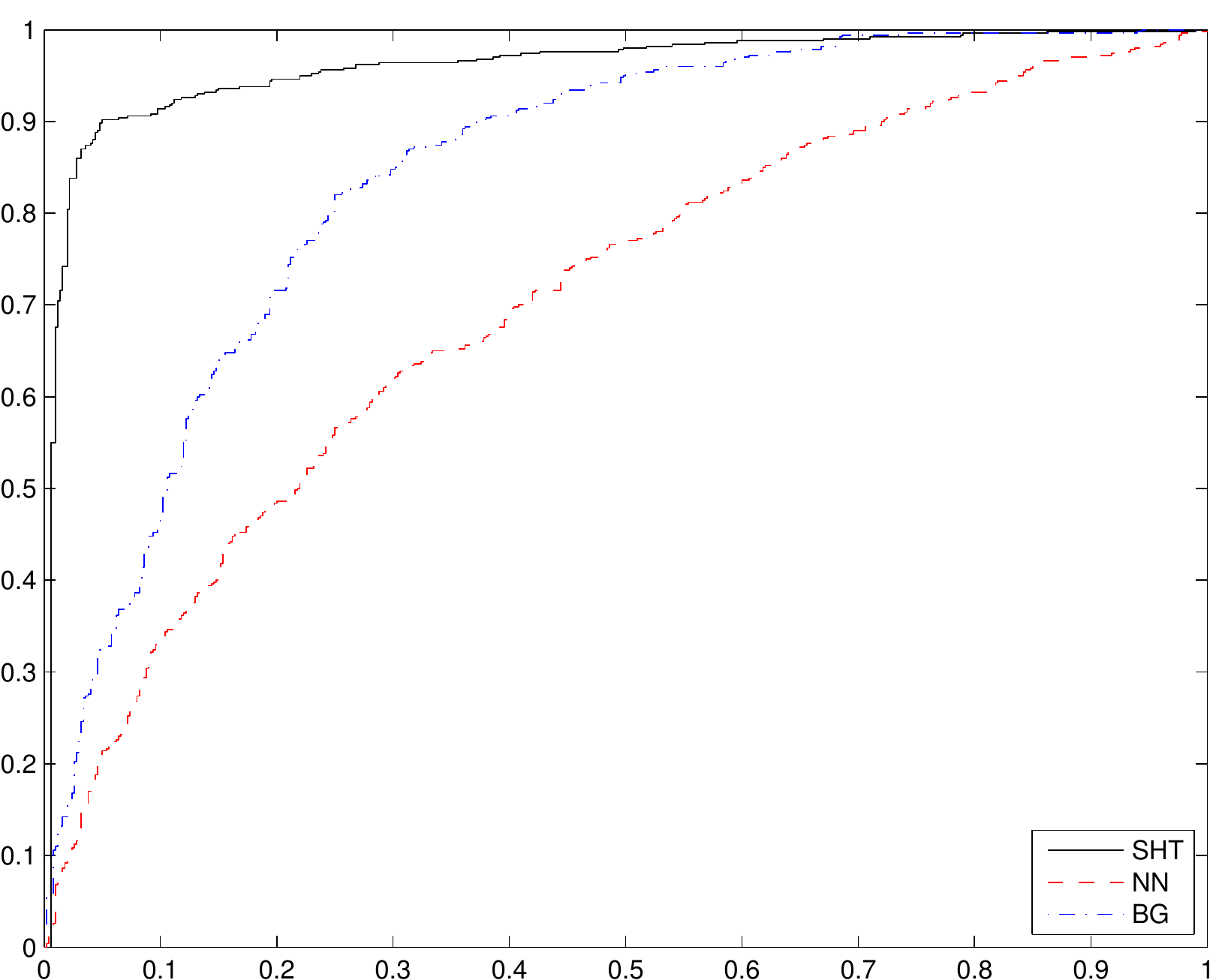}\\
a & b 
\end{tabular}
\caption{ROC Curves for the three methods and for the alternative $H_1^b$: a/ No noise and $N=100$. b/ Laplace noise with variance 0.1 and $N=100$.}
 \label{ROCH1d}
\end{center}\end{figure}

We precise that on Figure \ref{ROCH1d}a the solid line corresponding to the performance of our procedure SHT is mixed up with the axes passing through the points $(0,0), (0,1)$ and $(1,1)$.

\subsection{Real data: Paleomagnetism}

Some minerals in rocks have the particular property of conserving the direction of magnetic field when they cool down. This allows geologists to retrieve the direction of the Earth's magnetic field in the past ages, and this also provides information about the past location of tectonic plates. Since these records of magnetic field are directions, it is about spherical data. According to the accuracy of the measuring devices, the  measurements can be more or less noisy. To illustrate our method, we use the data given by \cite{FisherLewisEmbleton}: these are 52 measurements of magnetic remanence from specimens of red beds from the Bowen Basin (Queensland, Australia), after thermal demagnetisation to 670°C. Demagnetization is a process which is used in order to eliminate unwanted magnetic fields.  For our illustration, the advantage of this data set is that the possible anisotropy is not visible to the naked eye, contrary to a lot of data sets in paleomagnetism, see Figure~\ref{Paleomag}. 

\begin{figure}[!h]\begin{center}
\includegraphics[scale=0.4]
{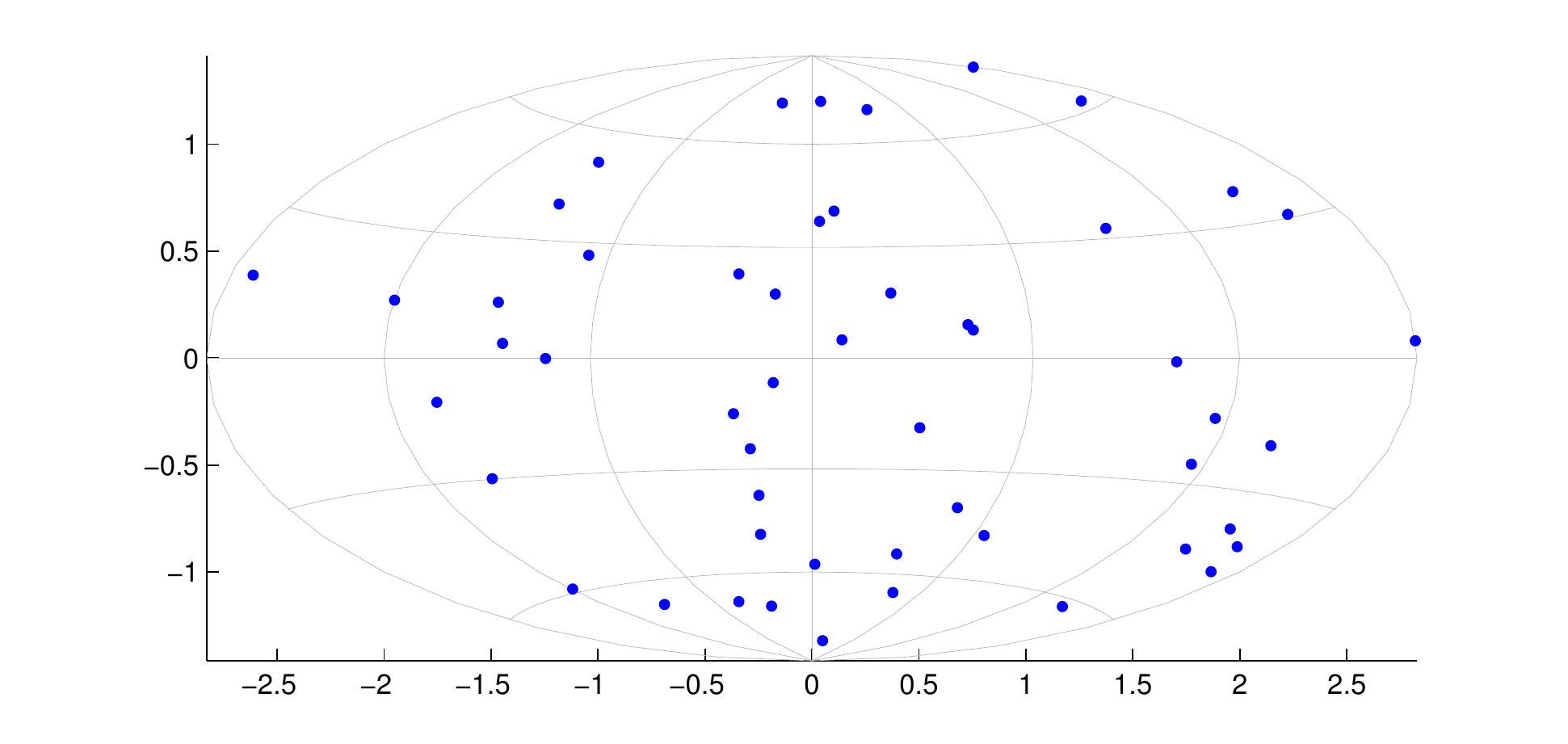}
\caption{Representation of the 52 measurements of magnetic remanence (Hammer projection) }
\label{Paleomag}
\end{center}\end{figure}

The obtained $p$-values for this sample are given in Table~\ref{pvalmagn}, assuming different kinds of possible noise. 
Whatever the possible noise we consider, there is no a statistical evidence of anisotropy, which indicates that the demagnetization process succeeded.

\begin{table}[!h]\begin{center}
\begin{tabular}{c|c|c c c| c c c}
Noise type & No noise & \multicolumn{3}{c}{Laplace} & \multicolumn{3}{c}{Gaussian}\\
variance & &  0.05  &  0.1  &  0.2 &  0.05 &  0.1 &  0.2 \\
\hline
$p$-value & 0.90 & 0.81 & 0.79 & 0 .77 & 0.73 & 0.63 & 0.74
\end{tabular}
\caption{$p$-values for the magnetism data.}
\label{pvalmagn}
\end{center}\end{table}

\subsection{Real data: UHECR}

To apply our procedure to UHECR data of observatory Pierre Auger (\cite{PA}), we need to take into account the observatory exposure. Indeed, only cosmic rays with zenith angle of arrival less that 60° can be observed. Then, a coverage function over the years of observation can be computed from geometrical considerations and it is displayed in Figure~\ref{UHECR}. 

\begin{figure}[!h]\begin{center}
\begin{tabular}{ccc}
\includegraphics[scale=0.4]
{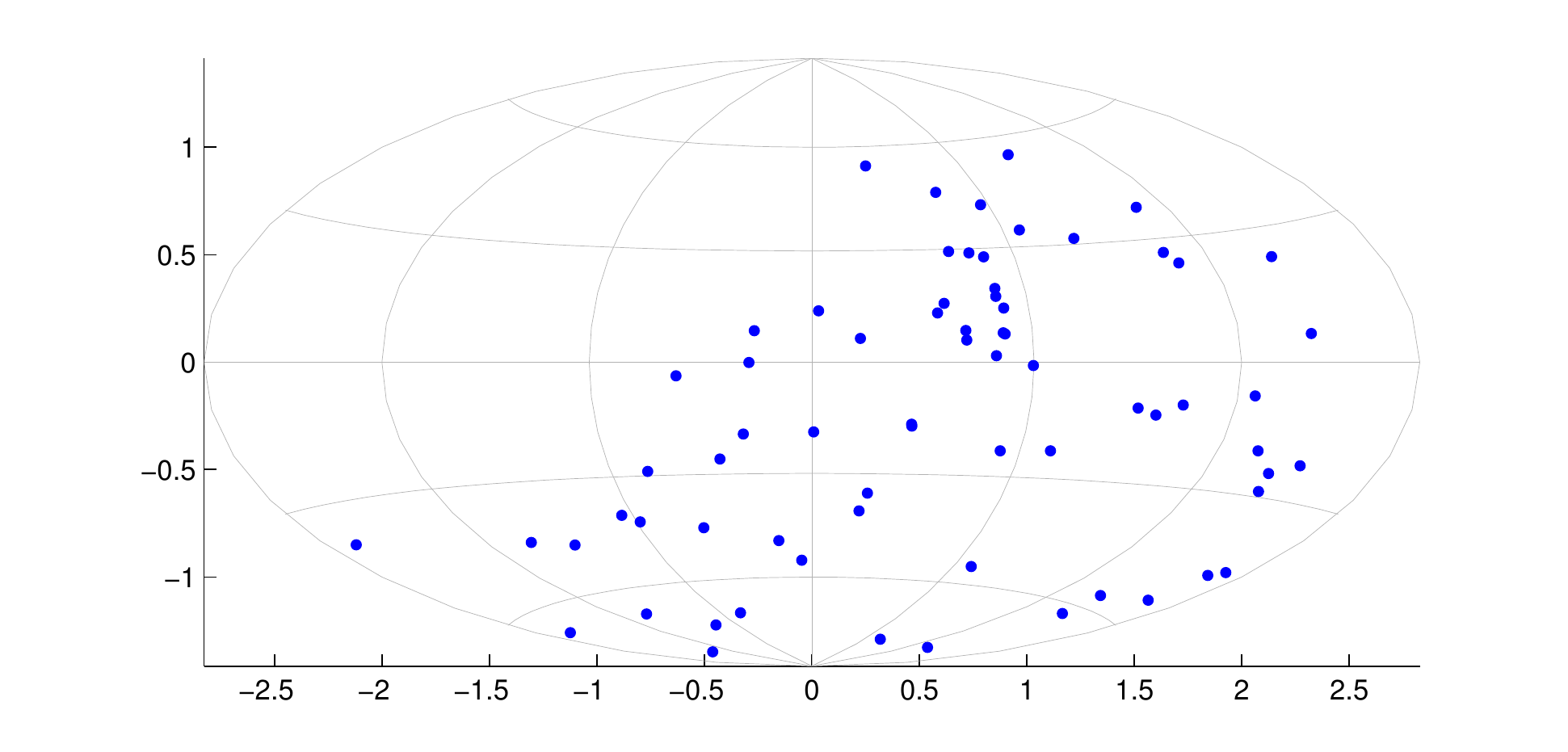}&
\includegraphics[scale=0.4]{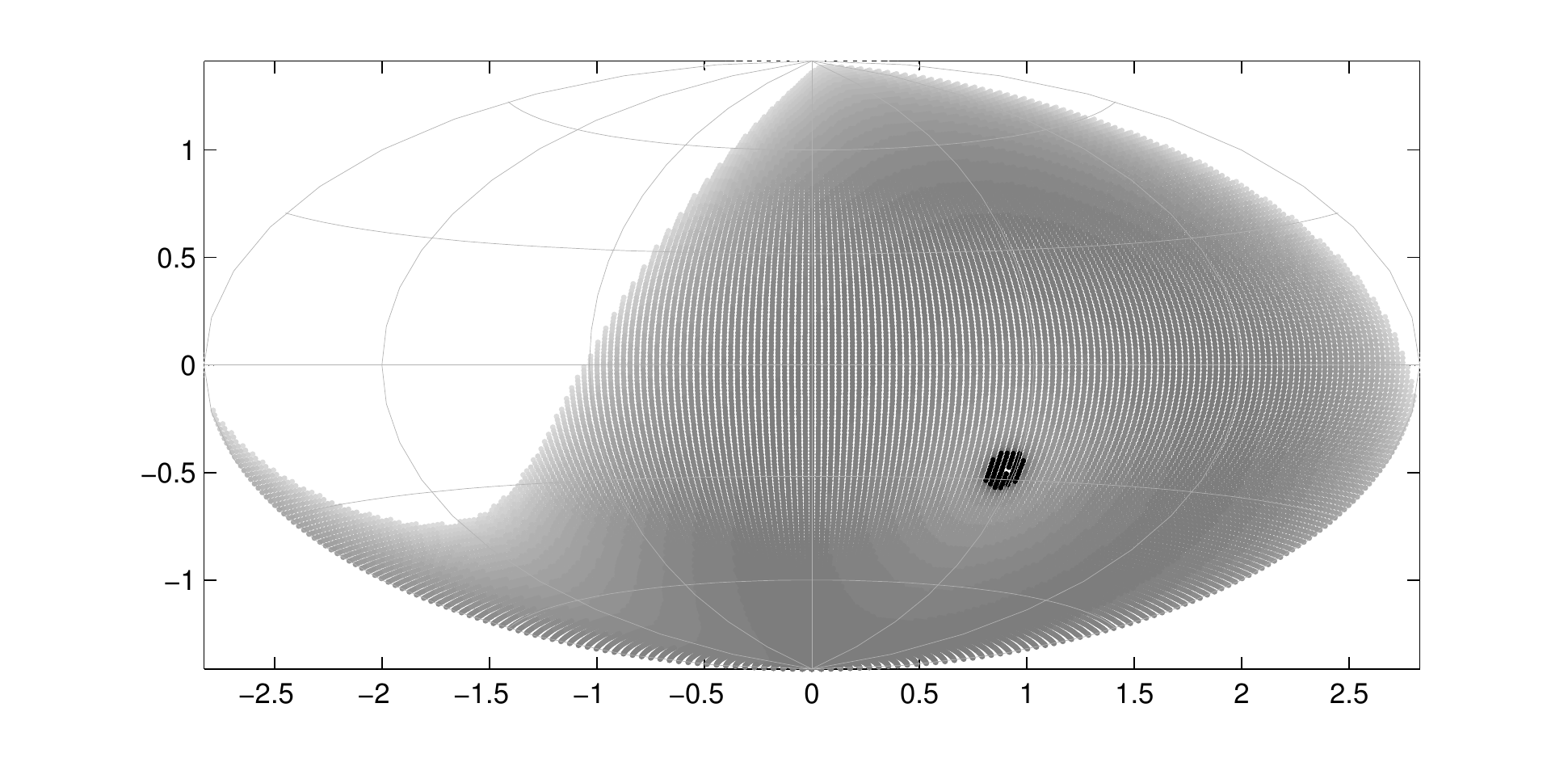}\\
a & b
\end{tabular}
\caption{a/ Representation of the 69 arrival directions of highest energy cosmic rays (Pierre Auger data) b/ Coverage function $g_{0}$ for the Pierre Auger observatory (the darker the more observed, white area non-observed)}
\label{UHECR}
\end{center}\end{figure}

In addition to the noise due to extragalactic magnetic fields, a selection is done depending on whether the ray is in the observation area. Denoting the coverage density by $g_{0}$, the observations are now $V_{1},\dots, V_{N}$ where the density of $V$ is proportional to $g_{0}$ times $f_{Z}$: $f_{V}=cg_{0}f_{Z}$, with $c$ such that $f_{V}$ is a density. The relevant test is then $f=f_{0}\Leftrightarrow f_{V}=g_{0}$. Although we do not  extend our theorems to this case, we nevetherless implement an extended method. Our initial test procedure is based on the estimation of $(f_{Z})^{*l}_{n}$ by $N^{-1}\sum_{i}\overline{Y^l_n(Z_i)}$. Then it is sufficient to apply the same procedure but with estimator 
$N^{-1}\sum_{i=1}^{N}({\overline{Y^l_n}}/(cg_{0}))(V_i).$ Indeed this quantity approximates $\int f_{V}\overline{Y^l_n}/(cg_{0}) =\int f_{Z}\overline{Y^l_n}=(f_{Z})^{*l}_{n}$.
Using this method, we obtain the $p$-values given in Table~\ref{pvalPA}, assuming different kind of possible noise. 

\begin{table}[!h]\begin{center}
\begin{tabular}{c|c|c c c| c c c}
Noise type & No noise & \multicolumn{3}{c}{Laplace} & \multicolumn{3}{c}{Gaussian}\\
variance & &  0.05  &  0.1  &  0.2 &  0.05 &  0.1 &  0.2 \\
\hline
$p$ & 0.003 & 0.014 & 0.034 & 0.092 & 0.016 & 0.001 & 0.076
\end{tabular}
\caption{$p$-values for the Pierre Auger data.}
\label{pvalPA}
\end{center}\end{table}

Then our method confirms what was already noticed by astrophycisists: there seems to be some kind of anisotropy in the UHECR phenomenon.

\section{Proofs}

\subsection{Proof of Theorem~\ref{bi}}

As usual in the proofs of lower bounds (see for instance \cite{Ingster00}, \cite{tsybakov09}), we build a set of functions quite far from $f_0$ in terms of the $\mathbb{L}_2$ norm, but whose distance between the resulting models is small. More precisely, let $\gamma=\gamma(N)$ and $L=L(N)$, respectively a scale factor and an even resolution level to be specified later. For all $l< L$ and $m\in \{-l,\dots, l\}$, we define
$\varphi_{lm}$ the function such that $$f_\varepsilon*\varphi_{lm}=Y_m^l.$$
Here $Y_m^l$ denotes the real form of the spherical harmonic, that we denote as the complex form for the sake of simplicity. Using the real form ensures that 
$\varphi_{lm}$ is real. The existence of such a function is ensured by the assumption of invertibility of matrices $f_\varepsilon^{\star l}$ and we can write 
$\varphi_{lm}=\sum_{n=-l}^l (f_{\eps^{-1}}^{\star l})_{nm}Y_n^l$. Now, for
$\theta_{lm}$, $l< L$, $m\in \{-l,\dots, l\}$, independent random variables with distribution $\P(\theta_{lm}=\pm\gamma)=1/2$,  we introduce 
$$f_\theta=f_0+\sum_{l=L/2}^{L-1}\sum_{m=-l}^l\theta_{lm}\varphi_{lm}.$$
In the sequel we show that, for good choices of $\gamma$ and $L$,  
\begin{itemize}
 \item $f_\theta$ belongs to $W_s(\S^2,R)$, 
 \item $f_\theta$ is a density function,
 \item $\|f_\theta-f_0\|^2\geq \mathcal{C}\psi_N$,
 \item $\chi^2(\P_\theta,\P_{f_0})\leq (1-\eta)^2\quad$ for $N$ large enough. 
\end{itemize}
For a definition of the $\chi^2$ divergence, see Section 2.4 in  \cite{tsybakov09}.\\
Then, for any test $T$, 
\begin{equation}\label{lb}
\P_{f_0}(T=1)+\P_\theta(T=0)\geq \int \min (d\P_\theta,d\P_{f_0})\geq 1-\sqrt{\chi^2(d\P_\theta,d\P_{f_0})}\geq 1-(1-\eta)
\end{equation}
using Lemma 2.1 (Scheff\'e's theorem) and inequality (2.27) in Section 2.4 in \cite{tsybakov09}. Thus, for $N$ large enough,
$$\P_{f_0}(T=1)+\sup_{f\in H_1(s,R,{\cal C}\psi_N)}\P_f(T=0)\geq \eta$$
and the result is proved. \\

\noindent$\bullet$\textit{Belonging to the Sobolev space: }\\ 
We compute
\begin{eqnarray*}
\sum_{l\geq 0}\sum_{ n=-l}^l (1+l(l+1))^{s} |<f_\theta, Y_n^l>|^2
&=& |<f_\theta, Y_0^0>|^2+\sum_{l=L/2}^{L-1}\sum_{n=-l}^{l} (1+l(l+1))^{s} |<f_\theta, Y_n^l>|^2\\
&\leq& |<f_0, Y_0^0>|^2+\sum_{l=L/2}^{L-1}(1+l(l+1))^{s} \sum_{n=-l}^{l}\left|\sum_{m=-l}^l\theta_{lm} (f_{\eps^{-1}}^{\star l})_{nm}\right|^2\\
&\leq &\frac1{4\pi}+\sum_{l=L/2}^{L-1}(1+l(l+1))^{s}\|f_{\eps^{-1}}^{\star l}\|^2_{op}\sum_{m=-l}^l|\theta_{lm}|^2 \\
&\leq &\frac1{4\pi}+d_0^{-2}\gamma^2\sum_{l=L/2}^{L-1}(1+l(l+1))^{s}l^{2\nu}(2l+1)\\
&\leq &\frac1{4\pi}+C_1(d_0,s,\nu)\gamma^2L^{2s+2\nu+2}.
\end{eqnarray*}
Hence, belonging to the Sobolev ball imposes that
$$  \gamma^2L^{2(s+\nu+1)} \leq \frac{R^2}{C_1(d_0,s,\nu)}.$$
Then it is sufficient to choose  $\gamma^2=c_1 L^{-2(\nu+s+1)}$ with $c_1\leq{R^2}/C_1(d_0,s,\nu)$.\\

\noindent$\bullet$\textit{Density: }\\ 
Since, for $l>0$,
$$\frac{1}{\sqrt{4\pi}}\int \varphi_{lm}=(\varphi_{lm})_0^{\star 0}=0$$
we obtain $\int f_\theta=1$. 
Let us now show that $f_\theta\geq 0$. 
We first use the Cauchy-Schwarz inequality
\begin{eqnarray*}
|f_\theta(x) -f_0(x)|^2&\leq &\sum_{l=L/2}^{L-1}\sum_{n=-l}^{l}|\langle f_\theta -f_0, Y_n^l\rangle|^2\sum_{l=L/2}^{L-1}\sum_{n=-l}^{l} |Y_n^l(x)|^2 \end{eqnarray*}
Next, since spherical harmonics have the property $\sum_{n=-l}^{l} |Y_n^l|^2 \leq (2l+1)/(4\pi)$, 
\begin{eqnarray*}
|f_\theta(x) -f_0(x)|^2&\leq & \sum_{l=L/2}^{L-1}\sum_{n=-l}^{l}\left|\sum_{m=-l}^l\theta_{lm} (f_{\eps^{-1}}^{\star l})_{nm}\right|^2
\sum_{l=L/2}^{L-1}\frac{2l+1}{4\pi}\\
&\leq &\frac{3}{8\pi}\sum_{l=L/2}^{L-1}\|f_{\eps^{-1}}^{\star l}\|^2_{op}\sum_{m=-l}^l|\theta_{lm}|^2 L^2 \\
&\leq &\frac{9d_0^{-2}}{8\pi}\gamma^2\sum_{l=L/2}^{L-1}l^{2\nu +1} L^2 
\leq C_2^2\gamma^2L^{2\nu +4}  
 \end{eqnarray*}
with $C_2^2={9d_0^{-2}}/({8\pi(2\nu +2)})$.
 Thus, replacing $\gamma$ by its value,
$$f_\theta\geq f_0-\|f_\theta- f_0\|_\infty\geq \frac1{4\pi}-C_2\gamma L^{\nu+2}\geq \frac1{4\pi}-C_2\sqrt{c_1}L^{1-s}.$$
Now, since $s\geq 1$, $f_\theta$ is a density as soon as 
$$\frac1{4\pi}-C_2\sqrt{c_1}\geq 0\Leftrightarrow c_1\leq \frac{\nu +1}{9\pi d_0^{-2}}.$$

\noindent$\bullet$\textit{Separation rate: }\\
We denote $p_\theta=f_\varepsilon*f_\theta=f_0+\sum_{l=L/2}^{L-1}\sum_{m=-l}^l\theta_{lm}Y_m^l$. Then
\begin{eqnarray*}
 \|p_\theta -f_0\|^2&=&\|f_\eps*(f_\theta-f_0)\|^2=\sum_{l=L/2}^{L-1}\sum_{m=-l}^l|(f_\eps^{\star l} (f_\theta-f_0)^{\star l})_m|^2
\\&\leq &\sum_{l=L/2}^{L-1}\|f_\eps^{\star l}\|_{op}^2 \sum_{m=-l}^l|(f_\theta-f_0)^{\star l}_m|^2
\leq d_1^2(L/2)^{-2\nu}\|f_\theta-f_0\|^2. 
\end{eqnarray*}
Moreover 
$$\|p_\theta -f_0\|^2=\sum_{l=L/2}^{L-1}\sum_{m=-l}^{l}|\theta_{lm}|^2=\gamma^2 \sum_{l=L/2}^{L-1} (2l+1)\geq C\gamma^2 L^2.$$
Finally
$\| f_\theta-f_0\|^2 \geq C_3(d_1,\nu)\gamma^2L^{2\nu+2} = C_3(d_1,\nu)c_1 L^{-2s}$.
Now we choose 
$$L=2\left\lfloor N^{1/(2s+2\nu+1)}\right\rfloor$$
where $\lfloor x\rfloor$ denotes the largest integer which is smaller than or equal to $x$.
Thus $\| f_\theta-f_0\|^2 \geq  C_3(d_1,\nu)c_1 2^{-2s}\psi_N\geq {\cal C}\psi_N$ as soon as ${\cal C}\leq C_3(d_1,\nu)c_12^{-2s}.$\\

 \noindent$\bullet$\textit{Chi-square divergence:}\\
We denote  $\mu$ the measure defined by $d\mu(\theta)=\prod_{l\geq 0, |m|\leq l}(\delta_1+\delta_{-1})/2.$ 
We want to show that 
$$ \E_{f_0}\left(\left(\frac{d\P_{\mu}}{d\P_{f_0}}-1\right)^2\right)\leq (1-\eta)^2$$
where $$\frac{d\P_{\mu}}{d\P_{f_0}}=\int \prod_{i=1}^N\frac{p_\theta(Z_i)}{p_0(Z_i)}\mu(d\theta)
=\E_{\mu}\prod_{i=1}^N{4\pi p_\theta(Z_i)}.$$
First, note that $\E_{f_0}(4\pi p_\theta(Z_1))=1+\sum_{l\geq 0, | m|\leq l}4\pi\theta_{lm}\int Y_m^l=1$.
Then, using Fubini and the independence of the $Z_i$, $\E_{f_0}(\frac{d\P_{\mu}}{d\P_{f_0}})=\E_{\mu}\prod_{i=1}^n\E_{f_0}(4\pi p_\theta(Z_i))=
1$.
So it is sufficient to prove that 
$$\E_{f_0}\left(\left(\frac{d\P_{\mu}}{d\P_{f_0}}\right)^2\right)\leq 1+(1-\eta)^2.$$
Using Fubini, 
\begin{eqnarray*}
\frac{1}{(4\pi)^{2N}} \E_{f_0}\left(\left(\frac{d\P_{\mu}}{d\P_{f_0}}\right)^2\right)
&=& \E_{f_0}\left(\E_{\mu\times\mu}\prod_{i=1}^Np_\theta(Z_i) \prod_{i=1}^Np_{\theta'}(Z_i)\right)
=
\E_{f_0}\left(\E_{\mu\times\mu}\prod_{i=1}^Np_\theta(Z_i)p_{\theta'}(Z_i)\right)\\
&=&\E_{\mu\times\mu}\left(\E_{f_0}\prod_{i=1}^Np_\theta(Z_i) p_{\theta'}(Z_i)\right)
\end{eqnarray*}
where $d(\mu\times\mu)((\theta,\theta'))=d\mu(\theta)d\mu(\theta')$.
Now, the $Z_i$ being independent, we write
\begin{eqnarray}\label{a}
 \E_{f_0}\left(\prod_{i=1}^Np_\theta(Z_i) p_{\theta'}(Z_i)\right)
=\prod_{i=1}^N\E_{f_0}(p_\theta(Z_i) p_{\theta'}(Z_i))
=\left(\int p_\theta p_{\theta'}\frac{1}{4\pi}\right)^N.
\end{eqnarray}
Using the definition of $p_\theta$ and the orthogonality of the $Y_m^l$, we obtain 
\begin{equation}\label{b}
 \int p_\theta p_{\theta'}=\frac1{4\pi}+\sum_{l=L/2}^{L-1}\sum_{m=-l}^l\theta_{lm}\theta'_{lm}.
\end{equation}
%
Combining \eqref{a} and \eqref{b}, and then inequality $1+a\leq e^a$, gives
\begin{eqnarray*}
 \E_{f_0}\left(\prod_{i=1}^Np_\theta(Z_i) p_{\theta'}(Z_i)\right)&=&\left(\frac1{(4\pi)^2}+\frac{1}{4\pi}\sum_{l=L/2}^{L-1}\sum_{m=-l}^l\theta_{lm}\theta'_{lm}\right)^N\\
&\leq& \frac1{(4\pi)^{2N}}\exp(N4\pi\sum_{l=L/2}^{L-1}\sum_{m=-l}^l\theta_{lm}\theta'_{lm})\\
&\leq&  \frac1{(4\pi)^{2N}}\prod_{l=L/2}^{L-1}\prod_{m=-l}^l\exp(N4\pi\theta_{lm}\theta'_{lm})
\end{eqnarray*}
so that 
\begin{eqnarray*}
  \E_{f_0}\left(\left(\frac{d\P_{\mu}}{d\P_{f_0}}\right)^2\right)
\leq \E_{\mu\times\mu} \prod_{l=L/2}^{L-1}\prod_{m=-l}^l\exp(N4\pi\theta_{lm}\theta'_{lm}).
\end{eqnarray*}
Using the distribution of $(\theta, \theta')$, we obtain
$$
  \E_{f_0}\left(\left(\frac{d\P_{\mu}}{d\P_{f_0}}\right)^2\right)
\leq \prod_{l=L/2}^{L-1}\prod_{m=-l}^l\frac12\exp(N\gamma^24\pi)+\frac12\exp(-N\gamma^24\pi)
\leq \prod_{l=L/2}^{L-1}\prod_{m=-l}^l\cosh(N\gamma^24\pi).
$$
Now, using $\cosh(2x)=1+2\sinh^2(x)$ and inequality $1+a\leq e^a$,
\begin{eqnarray*}
\E_{f_0}\left(\left(\frac{d\P_{\mu}}{d\P_{f_0}}\right)^2\right)&\leq& \prod_{l=L/2}^{L-1}\prod_{m=-l}^l(1+2\sinh^2(N\gamma^22\pi))
\leq \prod_{l=L/2}^{L-1}\prod_{m=-l}^l\exp(2\sinh^2(N\gamma^22\pi))\\
&\leq& \exp(2\sum_{l=L/2}^{L-1}\sum_{m=-l}^l\sinh^2(N\gamma^22\pi)).
\end{eqnarray*}
Since $N\gamma^2\to 0$ and $\sinh(x)=x+o(x)$, there exists $C_4>0$ such that, for $N$ large enough,
$$\sinh^2(N\gamma^22\pi)\leq C_4N^2\gamma^4.$$
Then 
$$\sum_{l=L/2}^{L-1}\sum_{m=-l}^l\sinh^2(N\gamma^22\pi)\leq C_4N^2\gamma^4\sum_{l=L/2}^{L-1}\sum_{m=-l}^l1\leq C_5N^2\gamma^4 L^2.$$
That yields, for $N$ large enough,
$$\E_{f_0}\left(\left(\frac{d\P_{\mu}}{d\P_{f_0}}\right)^2\right)\leq \exp(2C_5N^2\gamma^4L^2)
\leq \exp(2C_5c_1^2N^2L^{-4s-4\nu-2}).$$
But remember that $N^2<((L+1)/2)^{4s+4\nu+2}$ so that 
$\E_{f_0}\left(\left(\frac{d\P_{\mu}}{d\P_{f_0}}\right)^2\right)\leq \exp(C_6(s,\nu)c_1^2)\leq 1+(1-\eta)^2$ for a good choice of $c_1$.

\subsection{Proof of Theorem~\ref{biadap}}

We follow the same proof as the one of Theorem~\ref{bi} but  this time with a random $L$ (cf. \cite{Ingster00}). 
We choose $[s_{*},s^{*}]$ as an interval included in $\mathcal{S}\cap [1,\infty)$. Let $c_{0}=2\log(2)(2\nu+2s^{*}+1)^2$
and $k_N=\lfloor c_{0}^{-1} (s^{*}-s_{*})\log N \rfloor$. Then it is possible to choose $k_N$ elements of $\mathcal{S}$: $s_1<\dots <s_{k_N}$ such that 
$s_{j+1}-s_j\geq c_{0}/\log N$. Now we set, for $1\leq j \leq k_{N}$, 
$$J_{j}=\left\lfloor \frac{\log (N/\sqrt{\log\log N})}{(\log 2) (2\nu+2s_{j}+1)} \right\rfloor$$
which verifies  $J_1>\dots >J_{k_N}>1$ for $N$ large enough.
This choice ensures that $2^{J_j(2\nu+2s_j+1)}$ is of order $ N/\sqrt{\log\log N}$. 
We also define $\gamma_j^2=c_12^{-2J_j(\nu+s_j+1)}$.
We consider hypothesis functions  
$$f_\theta=f_0+\sum_{L}\sum_{l=L/2}^{L-1}\sum_{m=-l}^l\theta_{Llm}\varphi_{lm}$$
and we take a prior of the form 
$\mu={k_N}^{-1}\sum_{j=1}^{k_N}\mu_j$. Then  $\theta$ is randomly chosen such that 
$\mu_j(\theta_{Llm}=\pm \gamma_j)=1/2$ if $L=2^{J_j}, 2^{J_j-1}\leq l<2^{J_j}$, $-l\leq m\leq l$
and $\mu_j(\theta_{Llm}=0)=1$ otherwise.
This means that $L$ is fixed equal to $2^{J_j}$ with probability $1/k_N$ and random densities with respect to the measures $\mu_j$ have the following form
$$f_\theta=f_0+\sum_{l=L/2}^{L-1}\sum_{m=-l}^l\theta_{lm}\varphi_{lm}$$
where $\mu_j(\theta_{lm}=\pm \gamma_j)=1/2.$ 

Given the proof of Theorem~\ref{bi}, we easily verify that $\mu_j$-a.s. $f_\theta\in H_1(s_j,R,{\cal C}\psi_N^{ad}(s_j))$ if $c_1$ is chosen small enough. Now, 
since $\sup_{s\in\mathcal{S}}\sup_{f\in H_1(s,R, {\cal C}\psi_N^{ad}(s))}\P_f(\Delta_N=0)\geq \P_\mu(\Delta_N=0)$ and according to \eqref{lb}, it is sufficient to bound the  $\chi^2$-divergence.
So we will show that, for all $0<\eta <1$,
$$\limsup_N\E_{f_0}\left(\left(\frac{d\P_{\mu}}{d\P_{f_0}}\right)^2\right)\leq 1+(1-\eta)^2$$
which comes back to
$$ \limsup_N\frac1{k_N^2}\sum_{p,q=1}^{k_N}\E_{f_0}\left(\frac{d\P_{\mu_p}}{d\P_{f_0}}\frac{d\P_{\mu_q}}{d\P_{f_0}}\right)\leq 1+(1-\eta)^2.$$
Using Fubini's Theorem and independence of the $Z_i$'s, 
\begin{eqnarray*}
 \E_{f_0}\left(\frac{d\P_{\mu_p}}{d\P_{f_0}}\frac{d\P_{\mu_q}}{d\P_{f_0}}\right)
&=&\E_{\mu_p\times\mu_q}\left(\E_{f_0}\prod_{i=1}^Np_\theta(Z_i) p_{\theta'}(Z_i)\right)
=\E_{\mu_p\times\mu_q}\left(\left(\int 4\pi p_\theta p_{\theta'}\right)^N\right).
\end{eqnarray*}
Denoting $a_l(L_1,L_2)=4\pi\1_{L_1/2\leq l< L_1}\1_{L_2/2\leq l< L_2}$  we can write
$$ \int 4\pi p_\theta p_{\theta'}=1+\sum_{L_1,L_2, l, m}
\theta_{Llm}\theta'_{Llm}a_l(L_1,L_2)$$
where the sum is over $L_{1}\geq 0, L_{2}\geq 0, l\geq 0, |m|\leq l.$
Thus 
$$\left(4\pi\int p_\theta p_{\theta'}\right)^N\leq \exp\left(N\sum_{L_1L_2lm}\theta_{Llm}\theta'_{Llm}a_l(L_1,L_2)\right)
\leq \prod_{L_1L_2lm}\exp(N\theta_{Llm}\theta'_{Llm}a_l(L_1,L_2)).$$
Using the distribution of $(\theta, \theta')$, we obtain
\begin{eqnarray*}
\E_{f_0}\left(\frac{d\P_{\mu_p}}{d\P_{f_0}}\frac{d\P_{\mu_q}}{d\P_{f_0}}\right)
&\leq& \prod_{l\geq 0,|m|\leq l}\frac12\exp(N\gamma_p\gamma_q a_l(2^{J_p},2^{J_q}))+\frac12\exp(-N\gamma_p\gamma_qa_l(2^{J_p},2^{J_q}))
\\ &\leq &\prod_{l\geq 0,|m|\leq l}\cosh(N\gamma_p\gamma_qa_l(2^{J_p},2^{J_q})).
\end{eqnarray*}
Now, using $\cosh(2x)=1+2\sinh^2(x)$, inequality $1+a\leq e^a$ and $\sinh(x)=x+o(x)$
\begin{eqnarray*}
 \E_{f_0}\left(\frac{d\P_{\mu_p}}{d\P_{f_0}}\frac{d\P_{\mu_q}}{d\P_{f_0}}\right)
\leq \exp(2\sum_{l\geq 0,|m|\leq l}\sinh^2(N\gamma_p\gamma_qa_l(2^{J_p},2^{J_q})/2))
\leq \exp(C_1N^2\gamma_p^2\gamma_q^2\sum_{l\geq 0,|m|\leq l}|a_l(2^{J_p},2^{J_q})|^2).
\end{eqnarray*}
We observe that $a_l(2^{J_p},2^{J_q})=0$ as soon as $J_p\neq J_q$. 
That yields 
\begin{eqnarray*}
 \E_{f_0}\left(\frac{d\P_{\mu_p}}{d\P_{f_0}}\frac{d\P_{\mu_q}}{d\P_{f_0}}\right)
\leq \exp(C_2N^2\gamma_p^42^{2J_p}\1_{p=q}).
\end{eqnarray*}
Then 
\begin{eqnarray*}
 \frac1{k_N^2}\sum_{p,q=1}^{k_N}\E_{f_0}\left(\frac{d\P_{\mu_p}}{d\P_{f_0}}\frac{d\P_{\mu_q}}{d\P_{f_0}}\right)
 &\leq&\frac1{k_N^2}\sum_{p=1}^{k_N}\exp(C_2N^2\gamma_p^4 2^{2J_p})+\frac1{k_N^2}\sum_{p\neq q}1\\
&\leq &\frac1{k_N^2}\sum_{p=1}^{k_N}\exp(C_2c_1^2N^{2}2^{-2J_{p}(2\nu+2s_{p}+1)})+1-\frac1{k_{N}}\\
&\leq &\frac1{k_N}\exp(C_2c_1^22^{4\nu+4s^{*}+2} \log\log N)+1
\leq \frac{(\log N)^{C_3(\nu, s^*)c_1^2}}{k_N}+1
\end{eqnarray*}
which is bounded by $1+(1-\eta)^2$ for $N$ large enough if we choose $c_1$ small enough (smaller than $1/\sqrt{C_3}$). Since the bound is true for all $0<\eta<1$ and $c_1$ is chosen independently of $\eta$, this gives the result.


\subsection{Proof of Lemma~\ref{inegexp} }

We recall the result from \cite{ginelatalazinn}.

\begin{Lemma}\label{gine}
 Let $u$ denote a bounded canonical kernel, completely degenerate of the i.i.d. variables $Z_1,\dots, Z_N$. There exist universal constants $K_1, K_2>0$ such that, for all $x>0$, 
 $$\P\left(\left|\sum_{1\leq i_1\neq i_2\leq N}u(Z_{i_1}, Z_{i_2})\right|\geq x\right)\leq K_1\exp\left(-K_2\min\left(\frac{x^2}{C^2}, \frac{x}{D}, \frac{x^{2/3}}{B^{2/3}}, \frac{x^{1/2}}{A^{1/2}}\right)\right)$$
 where $A,B,C,D$ are defined by 
 $$A=\|u\|_\infty,\qquad B^2=N \|E(|u|^2(Z,.)\|_\infty,\qquad C^2=N^2\E[|u|^2(Z_1,Z_2)]$$ and 
 $$D=N\sup\left\{\big|\E[u(Z_1,Z_2)u_1(Z_1)u_2(Z_2)]\big|,\E[u_1^2(Z)]\leq 1,\E[u_2^2(Z)]\leq 1\right\}.$$
\end{Lemma}
We apply this Lemma to the  kernel 
$$u(x,y)=\sum_{l=1}^L\sum_{m=-l}^l\Phi_{lm}(x)\overline{\Phi_{lm}(y)}.$$
which is degenerate for $Z_i$ under $H_0$.
As one may have noticed, we stated the lemma above with a kernel $u$ taking complex values. Normally, the result of \cite{ginelatalazinn} was stated for real valued kernel. But their result can be extended to complex valued kernel by simply separating the real and imaginary parts as shown below. Indeed if we denote $u_R$ and $u_I$ the real and imaginary part of $u$ it follows that
\begin{eqnarray*}
&&\P\left(\left|\sum_{1\leq i_1\neq i_2\leq N}u(Z_{i_1}, Z_{i_2})\right|\geq x\right) = 
\P\left(\left|\sum_{1\leq i_1\neq i_2\leq N}u_R(Z_{i_1}, Z_{i_2})\right |^2 +\left |  \sum_{1\leq i_1\neq i_2\leq N}u_I(Z_{i_1}, Z_{i_2}) \right|^2\geq x^2\right) \\
&&\leq   \P\left(\left|\sum_{1\leq i_1\neq i_2\leq N}u_R(Z_{i_1}, Z_{i_2})\right |^2 \geq \frac{x^2}{2} \right) + \P \left( \left |  \sum_{1\leq i_1\neq i_2\leq N}u_I(Z_{i_1}, Z_{i_2}) \right|^2\geq \frac{x^2}{2}\right).
\end{eqnarray*}
Hence, it amounts to a real valued problem. 
We only deal with the real part since exactly the same arguments remain true for the imaginary part.  Let us show now that the bounds $A, B, C, D$ of Lemma~\ref{gine} hold for the real part $\P\left(\left|\sum_{1\leq i_1\neq i_2\leq N}u_R(Z_{i_1}, Z_{i_2})\right |^2 \geq{x^2}/{2} \right)$ . Because $u_R^2+u_I^2=|u|^2$, we have $u_R\leq |u_R| \leq |u|$.  Then 
\begin{eqnarray*}
\|u_R\|_{\infty} \leq \| u\|_{\infty} &\leq& A\\
N\| \E((u_R)^2(Z,.))\|_\infty  \leq N \|\E(|u|^2(Z,.))\|_\infty &\leq& B^2 \\
N^2\E[(u_R)^2(Z_1,Z_2)] \leq  N^2 \E[|u|^2(Z_1,Z_2)]&\leq& C^2.
\end{eqnarray*}
As for the last term $D$, since $u_1$ and $u_2$ are real valued
$$\left | \E (u(Z_1,Z_2)u_1(Z_1)u_2(Z_2) ) \right |^2=
 \left | \E (u_R(Z_1,Z_2)u_1(Z_1)u_2(Z_2) ) \right |^2  + \left | \E (u_I(Z_1,Z_2)u_1(Z_1)u_2(Z_2) )\right |^2 $$
it entails that 
$$ N\sup\left\{ \E (u_R(Z_1,Z_2)u_1(Z_1)u_2(Z_2) )\right\} 
 \leq N\sup\left\{ |\E (u(Z_1,Z_2)u_1(Z_1)u_2(Z_2) )|\right\} 
 \leq D$$
which concludes the justification our lemma. 
\\

Let us compute now the four bounds, $A,B, C, D$.\\

$\triangleright$ {\it Computation of $A$} \\
Denoting by $Y^l(x)$ the vector $(Y_m^l(x))_{-l\leq m\leq l}$ and using algebraic properties of the spherical harmonics,
\begin{equation}\label{aa}
 \sum_{m=-l}^l|\Phi_{lm}(x)|^2=\sum_{m=-l}^l|(f^{\star l}_{\varepsilon^{-1}}Y^l(x))_m|^2\leq \| f^{\star l}_{\varepsilon^{-1}}\|_{op}^2\sum_{m=-l}^l|Y_m^l(x)|^2
\leq d_0^{-2}l^{2\nu}\frac{2l+1}{4\pi}.
\end{equation}
We deduce  that
for all $x,y\in\mathbb{S}^2$,
$$|u(x,y)|\leq \sum_{l=1}^L\left(\sum_{m=-l}^l|\Phi_{lm}(x)|^2\sum_{m=-l}^l|\Phi_{lm}(y)|^2\right)^{1/2}
\leq \sum_{l=1}^L \pi^{-1} d_0^{-2}l^{2\nu+1}$$
so that $A\leq(\pi^{-1} d_0^{-2}/(2\nu+2))(L+1)^{2\nu+2}$.\\

$\triangleright$  {\it Computation of $C$} \\
We state the following Lemma, which allows to control the order of the variance of the test statistic.

\begin{Lemma}\label{ordrevar} Under Assumption 1, denoting $c_3=3d_0^{-4}2^{4\nu +2}/(4\nu +2)$
\begin{eqnarray*}\sum_{l_1,l_2=1}^L \sum_{m_1=-l_1}^{l_1}\sum_{m_2=-l_2}^{l_2}|\E_{f_0} (\Phi_{l_1m_1}(Z)\overline{\Phi_{l_2m_2}(Z)})|^2\leq c_3L^{4\nu+2}\\
\text{ and }\quad \sum_{l_1,l_2=1}^L \sum_{m_1=-l_1}^{l_1}\sum_{m_2=-l_2}^{l_2}|\E_{f_0} (\Phi_{l_1m_1}(Z)\Phi_{l_2m_2}(Z))|^2\leq c_3L^{4\nu+2}. 
\end{eqnarray*}
\end{Lemma}

\begin{proof}
 Under $H_0$, the $Z_i$ are uniformly distributed on the sphere. Then 
\begin{eqnarray*}
\E_{f_0}( \Phi_{l_1m_1}(Z)\overline{\Phi_{l_2m_2}(Z)})&=&\int \Phi_{l_1m_1}(z)\overline{\Phi_{l_2m_2}(z)}dz\\
&=&\sum_{n_1=-l_1}^{l_1}\sum_{n_2=-l_2}^{l_2}(f^{\star l_1}_{\varepsilon^{-1}})_{m_1n_1}(\overline{f^{\star l_2}_{\varepsilon^{-1}}})_{m_2n_2}\int \overline{Y^{l_1}_{n_1}(z)}Y^{l_2}_{n_2}(z)dz\\
&=&\sum_{n=-l_1}^{l_1}(f^{\star l_1}_{\varepsilon^{-1}})_{m_1,n}(\overline{f^{\star l_1}_{\varepsilon^{-1}}})_{m_2,n}\1_{l_1=l_2}.\\
\end{eqnarray*}
But, for any matrices $A=(a_{mn})_{-l\leq m\leq l,-l\leq n\leq l}$, $B=(b_{mn})_{-l\leq m\leq l,-l\leq n\leq l}$
$$\sum_{m_1=-l}^{l}|\sum_{n=-l}^{l} a_{m_1 n}b_{m_2 n}|^2\leq 
\| A\|_{op} ^2\sum_{n=-l}^{l}|b_{m_2 n}|^2\leq \| A\|_{op} ^2
\| B^T \|_{op} ^2=\| A\|_{op} ^2\| B \|_{op} ^2$$
Then 
\begin{equation}\label{ab}
 \sum_{m_1=-l_1}^{l_1}\sum_{m_2=-l_2}^{l_2}|\E_{f_0}( \Phi_{l_1m_1}(Z) \overline{\Phi_{l_2m_2}(Z)})|^2
\leq \| f^{\star l_1}_{\varepsilon^{-1}} \|_{op}^4 (2l_1+1)\1_{l_1=l_2}
\end{equation}
and 
\begin{eqnarray*}
\sum_{l_1, l_2=1}^L \sum_{m_1=-l_1}^{l_1}\sum_{m_2=-l_2}^{l_2}|\E_{f_0}( \Phi_{l_1m_1}(Z) \overline{\Phi_{l_2m_2}(Z)})|^2&\leq &\sum_{l_1=1}^L\| f^{\star l_1}_{\varepsilon^{-1}} \|_{op}^4 (2l_1+1)
\\ &\leq &3d_0^{-4} \sum_{l_1=1}^Ll_1^{4\nu +1}
\leq \frac{3d_0^{-4}}{4\nu +2}(L+1)^{4\nu +2}.
\end{eqnarray*}

For the second term, we can write, using \eqref{harmonique_negative}
\begin{eqnarray*}
\E_{f_0}( \Phi_{l_1m_1}(Z)\Phi_{l_2m_2}(Z))
&=&\sum_{n=-l_1}^{l_1}(f^{\star l_1}_{\varepsilon^{-1}})_{m_1,n}(f^{\star l_1}_{\varepsilon^{-1}})_{m_2,-n}(-1)^n\1_{l_1=l_2}.\\
\end{eqnarray*}
Then it is sufficient to apply the same method with matrix $B$ such that $b_{mn}=(-1)^n(f^{\star l_1}_{\varepsilon^{-1}})_{m,-n}=(-1)^na_{m,-n}$. The conclusion  
results from equality $\|B\|_{op}=\|A\|_{op}$.  
\end{proof}

Lemma \ref{ordrevar} gives $C^2\leq{3d_0^{-4}}/{(4\nu +2)} N^2(L+1)^{4\nu +2}$.\\

$\triangleright$  {\it Computation of $B$} \\
Let $x\in \mathbb{S}^2$ . We can write
\begin{eqnarray*}
 \E_{f_0}[|u(Z,x)|^2]=\sum_{l_1,l_2}\sum_{m_1,m_2} \E_{f_0}[ \Phi_{l_1m_1}(Z)\overline{\Phi_{l_2m_2}
(Z)}] \Phi_{l_2m_2}(x)\overline{\Phi_{l_1m_1}(x)}
\end{eqnarray*}
where the sum is over $ 1\leq l_{1}, l_{2}\leq L, |m_{1}|\leq l, |m_{2}|\leq l$.
But we have seen previously than $\E_{f_0}[ \Phi_{l_1m_1}(Z)\overline{\Phi_{l_2m_2}
(Z)}]$ vanishes when $l_1\neq l_2$. Then, using Cauchy-Schwarz inequality, we compute
\begin{eqnarray*}
 \E_{f_0}[|u(Z,x)|^2]\leq \sum_{l=1}^{L}\left(\sum_{-l\leq m_1,m_2 \leq l} |\E_{f_0}[ \Phi_{l m_1}(Z)\overline{\Phi_{l m_2}(Z)}]|^2 
 \sum_{-l\leq m_1,m_2\leq l}|\Phi_{l m_1}(x)\Phi_{l m_2}(x)|^2\right)^{1/2}.
\end{eqnarray*}
Now, we use previous computations \eqref{ab} and \eqref{aa} to state
\begin{eqnarray*}
 \E_{f_0}[u^2(Z,x)]&\leq &\sum_{l=1}^{L}\left(\| f^{\star l_1}_{\varepsilon^{-1}} \|_{op}^4 (2l+1)\right)^{1/2} \sum_{m=-l}^{l}|\Phi_{lm}(x)|^2\\
 &\leq &\sum_{l=1}^{L}\left(3d_0^{-4}l^{4\nu+1}\right)^{1/2} d_0^{-2}l^{2\nu}\frac{2l+1}{4\pi}\\
& \leq &\sum_{l=1}^L\frac{\sqrt{3}}{\pi}d_0^{-4} l^{4\nu+3/2}
\leq \frac{\sqrt{3}d_0^{-4}}{\pi(4\nu+5/2)}(L+1)^{4\nu +5/2}.
\end{eqnarray*}
Thus $B^2\leq\sqrt{3}\pi^{-1}d_0^{-4}/(4\nu+5/2)N(L+1)^{4\nu +5/2}$.\\

$\triangleright$  {\it Computation of $D$} \\
Let us first compute $\E_{f_0}( \Phi_{lm}(Z_1) u_1 (Z_1))$ under $H_0$. We denote by $U_1^l$ the vector of the Fourier coefficients of $u_1$ with harmonic order $l$:  $U_1^l=(<u_1,Y_n^l>)_{-l\leq n\leq l}$. 
\begin{eqnarray*}
 \E_{f_0}( \Phi_{lm}(Z_1) u_1 (Z_1))&=&\int \Phi_{lm}(x) u_1 (x) dx
 =\sum_{n=-l}^l (f^{\star l}_{\varepsilon^{-1}})_{mn}\int \overline{Y^l_n(x)}u_1(x)dx
 \\ &=&\sum_{n=-l}^l (f^{\star l}_{\varepsilon^{-1}})_{mn}<u_1,Y_n^l>
 =(f^{\star l}_{\varepsilon^{-1}}U_1^l)_m.
\end{eqnarray*}
Then 
$$\sum_{m=-l}^l| \E_{f_0}( \Phi_{lm}(Z_1) u_1 (Z_1))|^2=\|f^{\star l}_{\varepsilon^{-1}}U_1^l\|^2
\leq d_0^{-2}l^{2\nu}\|U_1^l\|^2.$$
But, using Parseval's equality
$$\sum_{l\geq 0}\|U_1^l\|^2= \sum_{l\geq 0}\sum_{n=-l}^l|<u_1,Y_n^l>|^2=\int u_1^2(x)dx$$
so that, under $H_0$, $\sum_{l}\|U_1^l\|^2\leq\E_{f_0}(u_1^2(Z_1))$.
In the same way we can prove 
$$\sum_{m=-l}^l|\E_{f_0}(\overline{ \Phi_{lm}(Z_2) }u_2 (Z_2))|^2\leq d_0^{-2}l^{2\nu}\|U_2^l\|^2$$ 
with $\sum_{l}\|U_2^l\|^2\leq\E_{f_0}(u_2^2(Z_1))$. 
Then, using repeatedly Cauchy-Schwarz inequality,
\begin{eqnarray*}
 \E_{f_0}(u(Z_1, Z_2)u_1(Z_1)u_2(Z_2))&=&
 \sum_{l=1}^L\sum_{m=-l}^l\E_{f_0}(\Phi_{lm}(Z_1)u_1(Z_1))\E_{f_0}(\overline{\Phi_{lm}(Z_2)}u_2(Z_2))\\
 &\leq &\sum_{l=1}^L\left(\sum_{m=-l}^l| \E_{f_0}( \Phi_{lm}(Z_1) u_1 (Z_1))|^2\sum_{m=-l}^l|\E_{f_0}( \overline{\Phi_{lm}(Z_2)} u_2 (Z_2))|^2\right)^{1/2}
\\
&\leq &\sum_{l=1}^Ld_0^{-2}l^{2\nu}\|U_1^l\|\|U_2^l\|
\leq d_0^{-2}L^{2\nu}\left(\sum_{l=1}^L\|U_1^l\|^2\sum_{l=1}^L\|U_2^l\|^2\right)^{1/2}\\
&\leq &d_0^{-2}L^{2\nu}\E_{f_0}^{1/2}(u_1^2(Z_1))\E_{f_0}^{1/2}(u_2^2(Z_2)).
\end{eqnarray*}
Thus $D\leq d_0^{-2}NL^{2\nu}$.\\

\noindent {\it Conclusion}\\
Now, using Lemma~\ref{gine} with $x=N(N-1) t/2$, we obtain
$$P\left(\left|T_L\right|\geq t\right)\leq K_1\exp\left(-K_3\min\left(\frac{N^2t^2}{L^{4\nu + 2}},\frac{Nt}{L^{2\nu}}, \frac{Nt^{2/3}}{L^{4\nu/3+5/6}}, \frac{Nt^{1/2}}{L^{\nu +1}}\right)\right)$$
where $K_3$ only depends on $d_0$ and $\nu$. Then 
\begin{eqnarray*}
 P\left(\left|T_L\right|\geq L^{2\nu+1}u_N/N\right)&\leq &K_1\exp\left(-K_3\min\left(u_N^2,u_N L, N^{1/3}u_N^{2/3} L^{-1/6}, N^{1/2}u_N^{1/2}L^{-1/2}\right) \right)\\
&\leq & K_1\exp(-K_0 u_N^2)
\end{eqnarray*}
provided that   $u_N=O(L)$, $L=O(N^2u_N^{-8})$ and $L=O(Nu_N^{-3})$.

\subsection{Proof of Theorem~\ref{bs}}

As in \cite{butuceatribouley}, we first use Lemma~\ref{inegexp}:
\begin{eqnarray*}
 \P_{f_0}(D_N=1)&\leq &\sum_{L\in\mathcal{L}}\P_{f_0}\left(|T_L|>\sqrt{2K_0^{-1}}t_L^2\right)
 \leq \sum_{L\in\mathcal{L}}\P_{f_0}\left(|T_L|>L^{2\nu+1}\sqrt{2K_0^{-1}\log\log N}/N\right)\\
& \leq &\sum_{L\in\mathcal{L}} K_1\exp(-K_0(2K_0^{-1}\log\log N))
=K_1\sum_{L\in\mathcal{L}}\exp(-2\log\log N)\\
& \leq &K_2|\mathcal{L}|(\log(N))^{-2} =O(\log(N)^{-1})=o(1)
\end{eqnarray*}
since $|\mathcal{L}|=O(\log(N))$.

Now let $f\in H_1(s, R,\C\psi_N)$. Then 
$$\P_f(D_N=0)=\P_f\left(\forall L\in\mathcal{L},\; |T_L|\leq \sqrt{2K_0^{-1}} t_L^2\right)
\leq \P_f\left( |T_{L^*}|\leq \sqrt{2K_0^{-1}} {t}^2_{L^*}\right)$$
with $L^*=2^{j*}$ and $j*=\lfloor \log_2[(N/\sqrt{\log\log N})^{1/(2s+2\nu+1)}]\rfloor$. 
Remark that for $N$ large enough, \\
$4\log\log N\leq (N/\sqrt{\log\log N})^{1/(2s+2\nu+1)} \leq N/(\log\log N)^{3/2}$, so that 
$j_0\leq j*\leq j_m$ and $L^*$ belongs to $\mathcal{L}$. 
Note also that with this choice 
$t^2_{L^*}\leq \psi_N$ and $L^{*-2s}\leq 2^{2s}\psi_N$.
 Using triangle inequality we have that
 \begin{equation}\label{AvantMarkov}
\P_f\left( |T_{L^*}|\leq \sqrt{2K_0^{-1}} {t}^2_{L^*}\right) \leq \P_f\left( | T_{L^*} -\E_f(T_{L^*}) | \geq \| f-f_0 \|^2_2 -\sqrt{2K_0^{-1}}t^2_{L^*} - B_f(T_{L^*})\right)
\end{equation}
where $B_f(T_{L})=\| f-f_0 \|^2_2-\E_f(T_{L})$.
If $f$ is in the Sobolev ball $W_s(\S^2,R)$, it directly follows from the definition of $W_s(\S^2,R)$ (\ref{Sobolev}) that
\begin{equation*}
 B_f(T_{L^{*}})=\sum_{l>L^{*}}\sum_{m=-l}^{l} |f^{\star l}_{m}| ^2\leq ((4\pi)^{-1}+R^2)L^{*-2s}
 \leq ((4\pi)^{-1}+R^2)2^{2s}\psi_N.
\end{equation*}
We set $C_1=\sqrt{2K_0^{-1}}+((4\pi)^{-1}+R^2)2^{2s}$ and $C_2= 1-C_1/\C>0$. 
Using the definition of $H_1$
$$\psi_N\leq \C^{-1} \| f-f_0 \|^2_2 .$$
Markov inequality yields the following upperbound for the expression (\ref{AvantMarkov})
\begin{equation}\label{bornerisque2}
 \P_f(D_N=0) \leq  \frac{\var_f(T_{L^*}) }  {   C_2^2\|f-f_0\|^4}.
  \end{equation}
%
Let us now state the following Lemma which evaluates the variance of the estimator $T_L$. 

\begin{Lemma}\label{var}
If Assumption 1 is verified, $$\var_f(T_L)\leq
c_4\left(\frac{L^{4\nu +2}}{N^2}+\frac{\|f-f_0\|^2L^{4\nu +4}}{N^2}+\frac{\|f-f_0\|^2L^{2\nu +1}}{N}+\frac{\|f-f_0\|^3L^{2\nu +2}}{N}+\frac{\|f-f_0\|^4}{N}\right)$$
where $c_4$ only depends on $d_0, d_1$ and $\nu$. 
\end{Lemma}

\begin{proof}
We have
$$\var_f(T_L)=\E((T_L-\E(T_L))(\overline{T_L- \E(T_L)}).$$
Simple calculations entail that
\begin{eqnarray*}
&&\var_f(T_L)= - \sum_{l_1,l_2=1}^L \sum_{m_1=-l_1}^{l_1}\sum_{m_2=-l_2}^{l_2}\ | f^{\star l_1}_{m_1}|^2 | f^{\star l_2}_{m_2}|^2 \\
&&+ \frac{4}{(N(N-1))^2} \left[ \sum_{l_1,l_2=1}^L \sum_{m_1=-l_1}^{l_1}\sum_{m_2=-l_2}^{l_2}\E\left (\sum_{i_1<i_2} \Phi_{l_1m_1}(Z_{i_1})\overline{\Phi_{l_1m_1}(Z_{i_2}}) \sum_{i_3<i_4}\overline{ \Phi_{l_2m_2}(Z_{i_3})}\Phi_{l_2m_2}(Z_{i_4})\right ) \right ]
\end{eqnarray*}
where $i_{1},i_{2},i_{3},i_{4}$ belong to $\{1,\dots, n\}$.
The term 
\begin{eqnarray*}
\E\left (\sum_{i_1<i_2} \Phi_{l_1m_1}(Z_{i_1})\overline{\Phi_{l_1m_1}(Z_{i_2}}) \sum_{i_3<i_4}\overline{ \Phi_{l_2m_2}(Z_{i_3})}\Phi_{l_2m_2}(Z_{i_4})\right )
\end{eqnarray*}
is bounded by
\begin{eqnarray*}
&&  \sum_{i_1< i_2} \sum_{i_3< i_4} \E(\Phi_{l_1m_1}(Z_{i_1})\overline{\Phi_{l_1m_1}(Z_{i_2})}\overline{\Phi_{l_2m_2}(Z_{i_3})}\Phi_{l_2m_2}(Z_{i_4}))
\\
&&=\sum_{i_1< i_2} \sum_{i_3< i_4} \bigg [ | f^{\star l_1}_{m_1}|^2 | f^{\star l_2}_{m_2}|^2  \1_{i_1\neq i_2 \neq i_3 \neq i_4}
+ |\E (\Phi_{l_1m_1}(Z)\overline{\Phi_{l_2m_2}(Z)})|^2\1_{i_1=i_3,i_2=i_4}\\
&&+ |\E(\Phi_{l_1m_1}(Z)\Phi_{l_2m_2}(Z))|^2\1_{i_1=i_4,i_2=i_3}
+ \E(\Phi_{l_1 m_1}(Z)\overline{\Phi_{l_2 m_2}(Z)}) \overline{ f^{\star l_1}_{m_1}}  f^{\star l_2}_{m_2}\1_{i_1=i_3,i_2\neq i_4}\\
&&+\E(\overline{\Phi_{l_1 m_1}(Z)}\Phi_{l_2 m_2}(Z))f^{\star l_1}_{m_1}  \overline{ f^{\star l_2}_{m_2}} \1_{i_1\neq i_3,i_2= i_4}
+\E(\Phi_{l_1 m_1}(Z)\Phi_{l_2 m_2}(Z))\overline{f^{\star l_1}_{m_1}  }\overline{ f^{\star l_2}_{m_2}} \1_{i_1= i_4,i_2\neq  i_3}\\
&&+\E(\overline{\Phi_{l_1 m_1}(Z)\Phi_{l_2 m_2}(Z)})f^{\star l_1}_{m_1}   f^{\star l_2}_{m_2} \1_{i_1\neq i_4,i_2= i_3})\bigg].
\end{eqnarray*}
Eventually we get that
\begin{eqnarray}\label{varianceT}
\var_f(T_L)&=&\sum_{l_1,l_2=1}^L\sum_{m_1=-l_1}^{l_1}\sum_{m_2=-l_2}^{l_2}\bigg[\left(\frac{(N-2)(N-3)}{N(N-1)} -1\right)| f^{\star l_1}_{m_1}|^2 | f^{\star l_2}_{m_2}|^2  \\
&&+\frac{1}{N(N-1)} \left(|\E (\Phi_{l_1m_1}(Z)\overline{\Phi_{l_2m_2}(Z)})|^2+|\E(\Phi_{l_1m_1}(Z)\Phi_{l_2m_2}(Z))|^2\right)\nonumber\\
&&+\frac{2(N-2)}{N(N-1)}\E(\Phi_{l_1 m_1}(Z)\overline{\Phi_{l_2 m_2}(Z)}) \overline{ f^{\star l_1}_{m_1}}  f^{\star l_2}_{m_2}\nonumber\\
&&+\frac{2(N-2)}{N(N-1)}\mathfrak{R}\left(\E(\Phi_{l_1 m_1}(Z)\Phi_{l_2 m_2}(Z))\overline{f^{\star l_1}_{m_1}  }\overline{ f^{\star l_2}_{m_2}}\right)
\nonumber\bigg]
\end{eqnarray}
where $\mathfrak{R}(x)$ denotes the real part of $x$. 
We shall now upperbound each term that appears in the expression \eqref{varianceT} above. \\

$\triangleright$ First term. Since $\sum_{l=1}^L \sum_{m} | f^{\star l}_{m} |^2\leq \|f-f_0\|^2$, we obtain
$$\sum_{l_1,l_2=1}^L \sum_{m_1=-l_1}^{l_1}\sum_{m_2=-l_2}^{l_2}\left(\frac{(N-2)(N-3)}{N(N-1)} -1\right)| f^{\star l_1}_{m_1}|^2 | f^{\star l_2}_{m_2}|^2  \leq \frac{\|f-f_0\|_2^4}N.$$

$\triangleright$ Second term. Firstly
\begin{eqnarray*}
 |\E(\Phi_{l_1m_1}(Z)\overline{\Phi_{l_2m_2}(Z)})|^2&=&
\left|\int \Phi_{l_1m_1}\overline{\Phi_{l_2m_2}} f_0+\int \Phi_{l_1m_1}\overline{\Phi_{l_2m_2}}(f_Z-f_0)\right|^2\\
&\leq & 2  |\E_{f_0}(\Phi_{l_1m_1}(Z)\overline{\Phi_{l_2m_2}(Z)})|^2+2\left|\int \Phi_{l_1m_1}\overline{\Phi_{l_2m_2}}(f_Z-f_0)\right|^2\\
&\leq & 2  |\E_{f_0}(\Phi_{l_1m_1}(Z)\overline{\Phi_{l_2m_2}(Z)})|^2+2\|\Phi_{l_1m_1} \overline{\Phi_{l_2m_2}}\|^2_2\|f_Z-f_0\|^2_2.
\end{eqnarray*}
We can remark that, under Assumption 1, 
\begin{eqnarray*}
 \|f_Z-f_0\|^2_2=\sum_{l\geq 0, |m|\leq l}|(f_Z-f_0)^{\star l}_m|^2\leq \sum_{l\geq 0} \| f_\eps^{\star l} \|_{op} ^2
 \sum_{|m|\leq l }|(f-f_0)^{\star l}_m|^2\leq d_1^2\|f-f_0\|^2_2
\end{eqnarray*}
since $\| f_\eps^{\star l}\|_{op} \leq d_1$ for all $l$.
Now let us show that there exists $C_1>0$ such that
$$\sum_{l_1,l_2=1}^L \sum_{m_1=-l_1}^{l_1}\sum_{m_2=-l_2}^{l_2}\|\Phi_{l_1m_1}\overline{ \Phi_{l_2m_2}}\|^2_2\leq C_1 L^{4\nu +4}.$$
We deduce from \eqref{aa} that
 \begin{eqnarray*}
\sum_{m_1=-l_1}^{l_1}\int |\Phi_{l_1m_1}(x)|^2 |\Phi_{l_2m_2}(x)|^2  dx&\leq &
\frac{3d_0^{-2}}{4\pi}l_1^{2\nu +1}
\int | \Phi_{l_2m_2}|^2 \leq \frac{3d_0^{-2}}{4\pi} l_1^{2\nu +1}
\sum_{m=-l_2}^{l_2} |(f_{\eps^{-1}}^{*l_2})_{m_2 m}|^2\\&\leq &\frac{3d_0^{-4}}{4\pi}l_1^{2\nu +1}l_2^{2\nu}.
\end{eqnarray*}
 Then 
 $$\sum_{l_1,l_2=1}^L \sum_{m_1=-l_1}^{l_1}\sum_{m_2=-l_2}^{l_2}\|\Phi_{l_1m_1} \overline{\Phi_{l_2m_2}}\|^2_2 \leq \frac{3d_0^{-4}}{4\pi}\sum_{l_1,l_2=1}^L \sum_{m_2=-l_2}^{l_2}l_1^{2\nu +1}l_2^{2\nu}\leq C_1 L^{4\nu +4}$$
and, using Lemma~\ref{ordrevar}, 
\begin{eqnarray*}
 \sum_{l_1,l_2=1}^L \sum_{m_1  ,m_2} |\E (\Phi_{l_1m_1}(Z)\overline{\Phi_{l_2m_2}(Z)})|^2&\leq &2\sum_{l_1,l_2=1}^L \sum_{m_1  ,m_2} |\E_{f_0}(\Phi_{l_1m_1}(Z)\overline{\Phi_{l_2m_2}(Z)})|^2\\&&+ 2C_1L^{4\nu +4}d_1^2\|f-f_0\|^2_2\\
&\leq & C_2(L^{4\nu +2} + L^{4\nu +4}\|f-f_0\|^2_2 ).
\end{eqnarray*}
In the same way 
\begin{eqnarray*}
 \sum_{l_1,l_2=1}^L \sum_{m_1=-l_1}^{l_1}\sum_{m_2=-l_2}^{l_2} |\E (\Phi_{l_1m_1}(Z)\Phi_{l_2m_2}(Z))|^2\leq  C_2(L^{4\nu +2} + L^{4\nu +4}\|f-f_0\|^2_2) .
\end{eqnarray*}
Thus, the second term is bounded by a constant times $L^{4\nu +2} /N^2+L^{4\nu +4}\|f-f_0\|^2_2 /N^2$.\\

$\triangleright$ Third term. Using Cauchy-Schwarz inequality we get
\begin{eqnarray*}
\sum_{l_1,l_2=1}^L \sum_{m_1  ,m_2}  f^{\star l_1}_{m_1}  \overline{ f^{\star l_2}_{m_2}}\E(\Phi_{lm_1}(Z)\overline{\Phi_{lm_2}(Z)})
&\leq & \left(\sum_{l_1,l_2=1}^L \sum_{m_1  ,m_2}  | f^{\star l_1}_{m_1}|^2 | f^{\star l_2}_{m_2}|^2  \right)^{1/2}
\\&&\left( \sum_{l_1,l_2=1}^L \sum_{m_1  ,m_2}|\E(\Phi_{l_1m_1}(Z)\overline{\Phi_{l_2m_2}(Z)})|^2 \right)^{1/2}\\
&\leq \sqrt{C_2} &\|f-f_0\|^2_2(L^{2\nu+1} + L^{2\nu +2}\|f-f_0\|_2).
\end{eqnarray*}
The third term is of order $ \|f-f_0\|^2_2 L^{2\nu +1}/N+\|f-f_0\|^3_2 L^{2\nu +2}/N$.\\

$\triangleright$ Fourth term. We bound the fourth term in the same way as the third.\\

Finally we have the bound for $\var_f(T_L)$
$$\frac{L^{4\nu +2}}{N^2}+\frac{\|f-f_0\|^2_2L^{4\nu +4}}{N^2}+\frac{\|f-f_0\|^2_2 L^{2\nu +1}}{N}+\frac{\|f-f_0\|^3_2 L^{2\nu +2}}{N}+\frac{\|f-f_0\|^4_2}{N}.$$
\end{proof}

This gives 
\begin{eqnarray*}
 \frac{\var_f(T_{L^*})}{\|f-f_0\|^4}&\leq &c_4\left(\frac{L^{*4\nu +2}}{N^2\|f-f_0\|^4}+\frac{L^{*4\nu +4}}{N^2\|f-f_0\|^2}+\frac{L^{*2\nu +1}}{N\|f-f_0\|^2}+\frac{L^{*2\nu +2}}{N\|f-f_0\|}+\frac1N\right).
 \end{eqnarray*}
Besides,  as $ \| f-f_0 \|^2_2 \geq \C2^{-2s}L^{*-2s}$ and $N\geq L^{*2s+2\nu +1}\sqrt{\log\log N}$, we get an upperbound for (\ref{bornerisque2}) in terms of $L^*$
\begin{eqnarray*}
C_3 \left(\frac1{C^2\log\log N}+ \frac{L^{*2-2s}}{C\log\log N}+\frac1{C\sqrt{\log\log N}}+\frac{L^{*1-s}}{\sqrt{C\log\log N}}+\frac1N\right).
\end{eqnarray*}
Since $s\geq 1$, all these terms tend to zero when $N$ goes to infinity,
and so does $\P_f(D_N=0)$.

Notice that this inequality gives a non asymptotic theoretical control of the second kind error of the test. Indeed this error is bounded by 
$C(s,\nu,d_0,d_1)(1-C_1/\C)^{-2}/\sqrt{\C}$, so it can be made lower than a fixed $\beta$, choosing $\C$ large enough.

\subsection{Proof of Theorem~\ref{bs2}}

This proof follows the same line as the one of Theorem~\ref{bs}. 
We first give an adaptation of  Lemma~\ref{var} in order to control the variance of $T_L$:
\begin{align}\begin{split}
\label{variance2}
{\rm Var}_f(T_L)&\leq C_0\left(\frac{L^{-4\nu_0 +2-\beta}}{N^2}e^{4L^\beta/\delta}
+\frac{\|f-f_0\|^2L^{-4\nu_0 +4-2\beta}}{N^2}e^{4L^\beta/\delta}\right.\\
&+\left.\frac{\|f-f_0\|^2L^{-2\nu_0 +1-\beta/2}}{N}e^{2L^\beta/\delta}+\frac{\|f-f_0\|^3L^{-2\nu_0 +2-\beta}}{N}e^{2L^\beta/\delta}
+\frac{\|f-f_0\|^4}{N}\right).
\end{split}\end{align}
This result is obtained with standard integrals evaluation which give for any real $\alpha$,
\begin{equation}\label{integ}
\sum_{l=1}^L l^\alpha e^{l^\beta/\delta}\leq C\int_1^{L+1} x^{\alpha} e^{x^\beta/\delta}dx\leq C' L^{\alpha+ 1-\beta}e^{L^\beta/\delta}
\end{equation}
(for $L$ large enough if $\alpha< 0$).
Now, we evaluate the first type error. Using that $\E_{f_0}(T_{L^*})=0$, we write
\begin{eqnarray*}
 \P_{f_0}(D_N=1)&= &\P_{f_0}\left(|T_{L^*}|>K_0t_{L^*}^2\right)
 \leq K_0^{-2}t_{L^*}^{-4}{\rm Var}_{f_0}(T_{L^*})\\
&\leq &K_0^{-2}C_0t_{L^*}^{-4}L^{*-4\nu_0 +2-\beta}\exp{(4L^{*\beta}/\delta)}N^{-2} \leq  K_0^{-2} C_0L^{*-\beta}=o(1)
\end{eqnarray*}
when $N$ goes to infinity.. To bound the error of the second kind, let $f\in H_1(s, R, \C\psi_N)$. We have
$$\P_f(D_N=0)\leq \P_f\left( |T_{L^*}|\leq K_0{t}^2_{L^*}\right)
\leq \P_f\left( | T_{L^*} -\E_f(T_{L^*}) | \geq \| f-f_0 \|^2_2 -K_0t^2_{L^*} - B_f(T_{L^*})\right).$$
The definition of $L^*$ implies that, 
 for $N$ large enough 
$$\left(\frac{\delta}{16} \log(N) \right)^{1/\beta}\leq L^*\leq\left(\frac{\delta}8 \log(N) \right)^{1/\beta}.$$
That ensures that $L^{*-2s}\leq (\delta/16)^{-2s/\beta}\psi_N$ and 
$t^2_{L^*}\leq (\delta\log N/8)^{(-2\nu_0+1)/\beta}N^{-3/4}\leq \psi_N$ for $N$ large enough. 
We set $C_1=K_0+((4\pi)^{-1}+R^2)(\delta/16)^{-2s/\beta}$ and $C_2= 1-C_1/\C$ (which is positive if $\C$ large enough). Markov inequality yields 
\begin{equation}\label{bornerisque3}
 \P_f(D_N=0) \leq  \frac{\var_f(T_{L^*}) }  {   C_2^2\|f-f_0\|^4}.
  \end{equation}
Using \eqref{variance2}, we bound
\begin{eqnarray*}
 \frac{\var_f(T_{L^*})}{\|f-f_0\|^4}&\leq &C_0\left(\frac{L^{*-4\nu_0 +2-\beta}}{N^2\|f-f_0\|^4}e^{(4L^{*\beta}/\delta)}+\frac{L^{*-4\nu_0 +4-2\beta}}{N^2\|f-f_0\|^2}e^{(4L^{*\beta}/\delta)}\right.\\
&&\left.+\frac{L^{*-2\nu_0 +1-\beta/2}}{N\|f-f_0\|^2}e^{(2L^{*\beta}/\delta)}+\frac{L^{*-2\nu_0 +2-\beta}}{N\|f-f_0\|}e^{(2L^{*\beta}/\delta)}+\frac1N\right).
 \end{eqnarray*}
Besides,  as $ \| f-f_0 \|^2_2 \geq C_3L^{*-2s}$, we get the following upperbound 
\begin{eqnarray*}
&&C_4\left(\frac{(\log N)^{(-4\nu_0 +2-\beta+4s)/\beta}N^{1/2}}{N^2}+\frac{(\log N)^{(-4\nu_0 +4-2\beta+2s)/\beta}N^{1/2}}{N^2}\right.\\
&&\left.+\frac{(\log N)^{(-2\nu_0 +1-\beta/2+2s)/\beta}N^{1/4}}{N}+\frac{(\log N)^{(-2\nu_0 +2-\beta+s)/\beta}N^{1/4}}{N}
+\frac1N\right).
\end{eqnarray*}
and all these terms tend to zero when $N$ goes to infinity.

\subsection{Proof of Theorem~\ref{bi2}}

The proof is analogous to the proof of Theorem~\ref{bi}, with hypothesis functions
$$f_\theta=f_0+\sum_{m=-L}^L\theta_{Lm}\varphi_{Lm},\qquad \P(\theta_{Lm}=\pm \gamma)=1/2,$$
where $$\gamma^2=c_1 \exp(-2L^\beta/\delta)L^{-2s+2\nu_0-1}$$
and $$L=\left\lfloor (2\delta\log(N))^{1/\beta}\right\rfloor.$$
This choice of $L$ ensures that, for $N$ large enough,
$$(\delta\log(N))^{1/\beta}\leq L \leq (2\delta\log(N))^{1/\beta}.$$
The four steps of the proof of Theorem~\ref{bi} can be rewritten. Moreover,  in this supersmooth case, the bound on the chi-square divergence is stronger,  so $c_1$ can be chosen  independently of $\eta$.

\section*{Acknowledgements}
We would like to thank Erwan Le Pennec and Richard Nickl for interesting discussions and suggestions.

\bibliographystyle{apalike}
\bibliography{biblio}

 \end{document}